\documentclass[10pt]{amsart}

\usepackage{tensor,a4wide,graphicx,mathrsfs}
\usepackage[pagebackref]{hyperref}

\author[A.V. Bolsinov, V.S. Matveev, T. Mettler, and S. Rosemann]{Alexey V. Bolsinov, Vladimir S. Matveev, Thomas Mettler, and Stefan Rosemann}

\numberwithin{equation}{section}

\newcommand{\grad}{\operatorname{grad}}
\newcommand{\be}{\begin{equation}}
\newcommand{\ee}{\end{equation}}
\newcommand{\weg}[1]{}

\newcommand{\bq}{\begin{equation}}
\newcommand{\eq}{\end{equation}}

\renewcommand{\i}{\mathrm{i}}
\newcommand{\R}{\mathbb{R}}
\newcommand{\C}{\mathbb{C}}
\newcommand{\tr}{\operatorname{tr}}
\newcommand{\Id}{\mathrm{Id}}

\newcommand{\gnabla}[0]{\tensor[^g]{\nabla}{}}

\newtheorem{thm}{Theorem}[section]
\newtheorem{lem}[thm]{Lemma}
\newtheorem{cor}[thm]{Corollary}
\newtheorem{prop}[thm]{Proposition}
\theoremstyle{definition}

\theoremstyle{remark}
\newtheorem{remark}[thm]{Remark}


\date{October 31, 2013.}
\thanks{TM thanks Forschungsinstitut f\"ur Mathematik (\textsc{fim}) at ETH Z\"urich and Schweizerischer Nationalfonds \textsc{snf}  for financial support via the postdoctoral fellowship PA00P2\textunderscore 142053.  VM and SR thank  DFG  and FSU Jena for partial financial support. AB and TM are grateful to FSU Jena for financial support for several trips to Jena where a part of the writing for this article took place}
\address{Department of Mathematical Sciences,
 Loughborough University,
 LE11 3TU, UK }
 \email{A.Bolsinov@lboro.ac.uk} 
\address{Institute of Mathematics, Friedrich-Schiller-Universit\"at Jena, Jena, Germany.}
\email{vladimir.matveev@uni-jena.de}
\address{Mathematical Institute, University of Oxford, UK.}
\email{mettler@maths.ox.ac.uk}
\address{Institute of Mathematics, Friedrich-Schiller-Universit\"at Jena, Jena, Germany.}
\email{stefan.rosemann@uni-jena.de}

\title[K\"ahler surfaces admitting c-projective vector fields]{Four-dimensional K\"ahler metrics \\admitting c-projective vector fields}

\begin{document}

\begin{abstract}
A vector field on a K\"ahler manifold is called c-projective if its flow preserves the $J$-planar curves. We give a complete local classification of K\"ahler real 4-dimensional manifolds  that  admit an essential  c-projective vector field. {An important technical step is a local description of 4-dimensional c-projectively equivalent metrics of arbitrary signature. As an  application of our results we  prove   the natural analog of the classical Yano-Obata conjecture in the 
    pseudo-Riemannian  4-dimensional case.} 
\end{abstract}
\maketitle

\section{Introduction}
\subsection{Definitions, results and motivation}

Let $(M, J, \nabla)$ be a real $2n$-dimensional smooth manifold equipped with a complex structure $J\in \textrm{End}(TM)$ which is parallel w.r.t.~a torsion-free affine connection $\nabla$.
A \emph{$J$-planar curve}  is a regular curve $\gamma:I\subseteq \mathbb{R} \to M$
 such that  the $2$-plane spanned by $\dot{\gamma}$ and $J \dot\gamma$ is parallel along $\gamma$, i.e., \begin{align}
\nabla_{\dot{\gamma}}\dot{\gamma}\wedge\dot{\gamma}\wedge J\dot{\gamma}=0.  \label{eq:gengeo}
\end{align}
In the literature, $J$-planar curves are also called \emph{holomorphically planar} or simply $h$-\emph{planar} curves. The notion of $J$-planar curves is an analog of the notion of geodesics, since  geodesics could be defined as regular curves  satisfying
$$
\nabla_{\dot{\gamma}}\dot{\gamma}\wedge\dot{\gamma}=0.
$$
Note that contrary to geodesics, at every point and in every direction there exist infinitely many $J$-planar curves.

A vector field $v$ on  $(M,J,\nabla)$ is called \textit{c-projective} if its (possibly locally defined) flow maps $J$-planar curves to $J$-planar curves. It is easy to see that such a vector field $v$ {automatically } preserves the complex structure. The set of c-projective vector fields with respect to  $(\nabla,J)$  forms a finite-dimensional  real Lie algebra; the set of affine (i.e., $\nabla$- and $J$-preserving) vector fields form a  subalgebra of this Lie algebra.

{
 \begin{remark} Most classical  sources  use the name ``h-projective'' or ``holomorphically-projective'' for what we call ``c-projective''  in our paper. We also used ``h-projective'' in our previous publications \cite{MR2948791,MR2998672}. 
  Recently a group of geometers studying c-projective geometry  from different viewpoints decided to change the name from h-projective to c-projective, since a c-projective change of connections (see \eqref{eq:christtrafo}), though being complex in the natural sense,  is generically not holomorphic. \weg{ The prefix  ``c-'' is chosen to be reminiscent of ``complex-'' but is not supposed
to be pronounced nor regarded as such since ``complex projective'' is
already used differently in the literature.}
 \end{remark} }

In this article we study essential  c-projective vector fields of Levi-Civita connections of  metrics of arbitrary signature. We assume that the metric $g$ is  hermitian  w.r.t.~to $J$, i.e., $g(J., .)$ is  skew-symmetric which, in view of $\nabla^g J=0$ implies that $g$ is K\"ahler.   
  The word \emph{essential} above means that  the vector field is not an affine vector field, i.e., it does not preserve the Levi-Civita connection. In the four-dimensional {Riemannian}  case, a vector field $v$ being affine for a K\"ahler metric $g$ implies one of the following possibilities: 1) $v$ is Killing for $g$; 2) $v$ is a homothety vector field for $g$; 3) the manifold is locally the direct product of two surfaces with a definite metric on each of them and $v$ is the direct sum of homothetic vector fields on the components; 4) $g$ is flat.  In view of this fact, all nonessential c-projective vector fields are easy to describe, at least for Riemannian metrics.

Note also that if we forget about the metrics and speak about the local description of c-projective vector fields for $(M^4, J, \nabla)$ (where $\nabla$ is an affine torsion-free connection preserving $J$), then the local   description almost everywhere is fairly simple. Indeed, in a coordinate system $x^1,...,x^4$ such that the c-projective vector field is given by $v=\tfrac{\partial }{\partial x^1}$ the Christoffel symbols $\Gamma_{jk}^i$ of  any   connection $\nabla$
  {such that  $v$ is a c-projective vector field for $(J,\nabla)$} are given  by
 $$
 \Gamma^i_{jk}=\check\Gamma^i_{jk}+ \phi_j \delta^i_k + \phi_j \delta^i_k - \phi_s J^s_{\ j} J^i_{\ k}  - \phi_s J^s_{\ k} J^i_{\ j},
 $$
  where the functions $\check\Gamma^i_{jk}$ do not  depend on the variable $x^1$ and the functions $\phi_i$ are arbitrary.

The main result of our paper, see Theorems  \ref{thm:normalformwithv} and \ref{thm:allmetricsintheclass}, is a complete local  description of the triples $(g,J,v)$  (K\"ahler metric, complex structure, essential c-projective vector field)
in a neighbourhood of almost every point in the four-dimensional case. Although the precise statements are slightly lengthy we indeed provide an explicit description of the components of the metric, of the K\"ahler form $\omega(.,.)= g(J., .)$ and of the c-projective vector field in terms of elementary functions. The parameters in this description are almost  arbitrary numbers $\beta, c,C, c_1, c_2, d_1,d_2$ and, in certain cases, almost arbitrary smooth  nonzero function $G$ of one variable.

An important step in this article (that may be viewed as a separate result) is a local description of {\it c-projectively equivalent} (i.e., sharing the same $J$-planar curves) four-dimensional split-signature K\"ahler metrics. In the case where $g$ is positive definite, such a description is a special case of the results of Apostolov et al \cite{MR2228318}  on Hamiltonian $2$-forms, we give some  details in \S \ref{sec:localclass}.

Despite the fact that the problem of classifying c-projectively equivalent  pseudo-Riemannian K\"ahler metrics is interesting on its own, we have another motivation for allowing the metrics to have arbitrary signature: the method how we obtained the  description of the metrics admitting essential c-projective vector fields in fact requires the description of c-projectively equivalent  K\"ahler metrics of arbitrary signature.

We had the following  motivation to study this problem. First of all, the problem is classical. The notions of $J$-planar curves,  c-projective vector fields and  c-projectively equivalent metrics were introduced by Otsuki and Tashiro~\cite{MR0066024,MR0087181}. For a certain period of time  this theory was one of the main research directions of the Japanese ($\sim$1950-1970) and Soviet ($\sim$1970-1990) differential geometric schools. There are many publications and results in this theory  and especially in the theory of c-projective vector fields,  see for example  \cite{MR2998672} for a list of references.  Most  (actually, all we have found) results in the theory of essential c-projective vector fields are negative: under certain assumptions on the geometry of the manifold,  it is proved that an essential c-projective vector field cannot exist. As far as we know, before our paper there  existed  no explicit  examples of essential c-projective vector fields on  K\"ahler  manifolds  of nonconstant holomorphic curvature. Our paper provides all possible
examples in dimension four. Furthermore, using the methods of our paper it is possible to construct an example for every dimension and every signature of the metric. This problem will be studied in a forthcoming article.

The second motivation is that the problem is a natural complex version of a problem posed by  Sophus Lie in 1882. In \cite{Lie}, Sophus Lie explicitly asked to describe all two-dimensional metrics admitting projective vector fields. Recall that a vector field is \emph{projective} if its local flow sends unparametrised geodesics to unparametrised geodesics. From the context it is clear that S.~Lie stated the problem in the local setup and allowed pseudo-Riemannian metrics. In this setting, it has been solved only recently in \cite{MR2368987,MR2892455}. The problem we solve in this paper is, in fact,  just  the Lie problem if we replace ``geodesics'' by ``$J$-planar curves''.
The main idea of our approach is borrowed from \cite{MR2892455} and was also implicitly used in \cite{MR2368987}. A main difficulty in implementing this idea to the c-projective setting  was the lack of a local description of c-projectively equivalent K\"ahler  metrics of indefinite signature. As mentioned above, even if we are interested in Riemannian metrics only, the approach borrowed from \cite{MR2892455} requires the local description of c-projectively equivalent metrics of all signatures, see \S \ref{trick} for details.

A third motivation is that there is a recent interest in the theory of c-projectively equivalent metrics from the side of K\"ahler geometers. It appears that c-projectively equivalent K\"ahler metrics correspond to Hamiltonian two-forms introduced and  studied in \cite{MR2228318,MR2144249, MR2425136, MR2411469}. In this context, the existence of an essential c-projective vector field is a natural geometric condition and we plan to study the possible implications of metrics admitting such a vector field in the future.
\weg{
Fourth, two-dimensional metrics whose algebra of (real) projective vector fields is of submaximal dimension are called Darboux-superintegrable. These metrics are of considerable interest in mathematical physics and differential geometry. We hope that certain metrics from our list will also be of interest to the mathematical physics community, in particular the metrics from the L1-family (see Theorem \ref{thm:normalformwithv}). Furthermore, c-projectively homogeneous four-dimensional K\"ahler metrics might be of interest as well in the context of mathematical physics. }

{
Finally, a main motivation was to prove the natural pseudo-Riemannian generalisation of the Yano-Obata conjecture in the $4$-dimensional case (see Theorem  \ref{YOconjecture} below). }

\subsection{Main theorems}
Assume $(M^4, g, J)$ is a   K\"ahler surface (i.e., a real four-dimensional K\"ahler manifold) of arbitrary signature.
Let  $\hat{g}$ be a $J$-hermitian metric on $(M,J)$. We call the metrics $g$ and $\hat{g}$  \emph{c-projectively equivalent}, if $J$-planar curves  of $g$ are $J$-planar curves of $\hat g$. This condition automatically implies that $(M^4, \hat{g}, J)$ is also K\"ahler (which easily follows from the relation \eqref{eq:christtrafo} below).

If a vector field $v$ is c-projective for $(g,J)$, then the pullback of $g$ with respect to the flow of $v$ is a K\"ahler metric that is c-projectively equivalent to $g$. Moreover, $v$ is a c-projective vector field for every metric $\hat{g}$ in the equivalence class $[g]$ of all metrics that are c-projectively equivalent to $g$. We call $v$ essential for the class $[g]$, if $v$ is essential for some metric $\hat{g}\in [g]$.

Our first theorem classifies all local c-projective equivalence classes $[g]$ in (real) four dimensions admitting an  essential c-projective vector field $v$.

\begin{thm}\label{thm:normalformwithv}
Let $(M,g_0,J)$ be a K\"ahler surface  of non-constant holomorphic curvature admitting an essential c-projective vector field $v$. Then in a neighbourhood of almost every point of $M$ there are local coordinates such that a certain metric $g \in [g_0]$, its K\"ahler form $\omega=g(J.,.)$ and $v$ are given by one of the cases L1--L4,CL1--CL4,D1--D3 described below.

Liouville case: There are coordinates $x,y,s,t$ and functions $\rho(x)$, $\sigma(y)$, $F(x)$, $G(y)$ of one variable such that $g$ and $\omega$ take the form
\begin{align}
\begin{array}{c}
g=(\rho-\sigma)(F^2 dx^2+\epsilon G^2 dy^2)+\frac{1}{\rho-\sigma}\left[\left(\frac{\rho'}{F}\right)^2\left(ds+\sigma dt\right)^2+\epsilon\left(\frac{\sigma'}{G}\right)^2\left(ds+\rho dt\right)^2\right],\vspace{1mm}\\
\omega=\rho'dx\wedge (ds+\sigma dt)+\sigma'dy\wedge ( ds+\rho dt),
\end{array}\nonumber
\end{align}
where $\epsilon=1$ in case of positive signature and $\epsilon=-1$ in case of split signature. The functions $\rho,\sigma,F,G$ and $v$ are as in the cases L1--L4 below.

\begin{itemize}
\item Case L1: $\rho(x)=x,\sigma(y)=y,F=c_1,G=c_2$ and $v=\partial_{x}+\partial_{y}-t\partial_s,$ where $c_1,c_2$ are constants.

\item Case L2: $\rho(x)=c_1\mathrm{e}^{(\beta-1)x},\sigma(y)=c_2\mathrm{e}^{(\beta-1)y},F(x)=d_1\mathrm{e}^{-\frac{1}{2}(\beta+2)x},G(y)=d_2\mathrm{e}^{-\frac{1}{2}(\beta+2)y}$ and $v=\partial_{x}+\partial_{y}-(\beta+2)s\,\partial_s-(2\beta+1)t\,\partial_t,$ where $\beta\neq 1$ and $c_1,c_2,d_1,d_2$ are constants.

\item Case L3: $\rho(x)=x,\sigma(y)=y,F(x)=c_1e^{-\frac{3}{2}x},G(y)=c_2e^{-\frac{3}{2}y}$ and $v=\partial_{x}+\partial_{y}-(3s+t)\partial_s-3t\,\partial_t,$ where $c_1,c_2$ are constants.

\item Case L4: $\rho(x)=-\tan(x),\sigma(y)=-\tan(y),F(x)=\frac{c_1\mathrm{e}^{-\frac{3}{2}\beta x}}{\sqrt{|\cos(x)|}},G(y)=\frac{c_2\mathrm{e}^{-\frac{3}{2}\beta y}}{\sqrt{|\cos(y)|}}$ and $v=\partial_{x}+\partial_{y}-(3\beta\, s-t)\partial_s-(s+3\beta\, t)\partial_t,$ where $\beta,c_1,c_2$ are constants.

\end{itemize}

Complex Liouville case: This case occurs in split signature only. There are coordinates  $z=x+\i y,s,t$ and holomorphic functions $\rho(z),F(z)$ such that $g$ and $\omega$ take the form
\begin{align}
\begin{array}{c}
g=\frac{1}{4}(\bar{\rho}-\rho)(F^2 dz^2-\bar{F}^2 d\bar{z}^2)+\frac{4}{\rho-\bar{\rho}}\left[\left(\frac{1}{\bar{F}}\frac{\partial\bar{\rho}}{\partial\bar{z}}\right)^2\left(ds+\rho\,dt\right)^2-\left(\frac{1}{F}\frac{\partial \rho}{\partial z}\right)^2\left(ds+\bar{\rho}\,dt\right)^2\right],\vspace{1mm}\\
\omega=\frac{\partial \rho}{\partial z}dz\wedge (ds+\bar{\rho}dt)+\frac{\partial \bar{\rho}}{\partial \bar{z}}d\bar{z}\wedge (ds+\rho dt).
\end{array}\nonumber
\end{align}
The functions $\rho,F$ and $v$ are as in the cases CL1--CL4 below.

\begin{itemize}
\item Case CL1: $\rho(z)=z,F(z)=c_1+\i c_2$ and $v=\partial_z+\partial_{\bar{z}}- t\,\partial_s,$ where $c_1,c_2$ are constants.

\item Case CL2: $\rho(z)=e^{(\beta-1)z},F(z)=(c_1+\i c_2)e^{-\frac{1}{2}(\beta+2)z}$ and $v=\partial_z+\partial_{\bar{z}}-(\beta+2)s\,\partial_s-(2\beta+1)t\,\partial_t,$ where $\beta\neq 1$ and $c_1,c_2$ are constants.

\item Case CL3: $\rho(z)=z,F(z)=(c_1+\i c_2)\mathrm{e}^{-\frac{3}{2}z}$ and $v=\partial_w+\partial_{\bar{w}}-(3s+t)\partial_s-3t\,\partial_t,$ where $c_1,c_2$ are constants.

\item Case CL4: $\rho(z)=-\tan(z),F(z)=\frac{(c_1+\i c_2)\mathrm{e}^{-\frac{3}{2}\beta z}}{\sqrt{\cos(z)}}$ and $v=\partial_z+\partial_{\bar{z}}-(3\beta\, s- t)\partial_s-(s+3\beta\,t)\partial_t,$ where $\beta,c_1,c_2$ are constants.

\end{itemize}

Degenerate case: There are coordinates $x,t,u_1,u_2$, functions $\rho(x),F(x)$ of one variable and a positive or negative definite $2D$ K\"ahler structure $(h,j,\Omega=h(j.,.))$ on the domain $\Sigma\subseteq \mathbb{R}^2$ of $u_1,u_2$ such that $g$ and $\omega$ take the form
\begin{align}
\begin{array}{c}
\hat{g}=-\rho h+\rho F^2 d x^2+\frac{1}{\rho}\left(\frac{\rho'}{F}\right)^2\theta^2,\,\,\,\hat{\omega}=-\rho \Omega+\rho'dx\wedge\theta,
\end{array}
\nonumber
\end{align}
where $\theta=dt-\tau$ and $\tau$ is a one-form on $\Sigma$ satisfying $d\tau=\Omega$. The functions $\rho,F$ and the forms $h,\tau$ are as in the cases D1--D3 below.

\begin{itemize}

\item Case D1: $\rho(x)=\tfrac{1}{x}$, $F(x)=\tfrac{c_1}{\sqrt{|x|}}$, $\tau=u_1 du_2$, $h=G(u_2)du_1^2+\tfrac{du_2^2}{G(u_2)}$ and $v=\partial_x+u_2\partial_t +\partial_{u_1}$, where $c_1$ is a constant and $G(u_2)$ is an arbitrary function.

\item Case D2: $\rho(x)=c_1e^{(\beta-1)x}$, $F(x)=d_1e^{-\frac{1}{2}(\beta+2)x}$ for certain constant $\beta\neq 1$, where $c_1,d_1$ are constants
\begin{itemize}
\item Subcase $\beta+2=0$: $\tau=u_1du_2$, $h=G(u_2)du_1^2+\frac{du_2^2}{G(u_2)}$ and $v=\partial_x+u_2 \partial_t +\partial_{u_1}$,
\item Subcase $\beta+2\neq 0$: $\tau =-\tfrac{1}{\beta+2}e^{-(\beta+2)u_1}G(u_2)du_2$, $h=e^{-(\beta+2)u_1}G(u_2)(du_1^2+du_2^2)$ and $v=\partial_x-(\beta+2)t\partial_t+\partial_{u_1}$,
\end{itemize}
where $G(u_2)$ is an arbitrary function.
\item Case D3: $\rho(x)=\tfrac{1}{x}$, $F(x)=\tfrac{c_1 e^{-\frac{3}{2}x}}{\sqrt{|x|}}$, $\tau=-\tfrac{1}{3}e^{-3 u_1}G(u_2)du_2$, $h=e^{-3 u_1}G(u_2)(du_1^2+du_2^2)$ and $v=\partial_x-3t\partial_t+\partial_{u_1}$, where $c_1$ is a constant and $G(u_2)$ is an arbitrary function.

\end{itemize}

Conversely, given local coordinates on $\mathbb{R}^4$ and $(g,\omega,v)$ as in one of the above cases, $(g,\omega)$ defines a K\"ahler structure (whenever the formulas make sense) and $v$ is an essential c-projective vector field for $[g]$.
\end{thm}

\begin{remark}
The listed metrics do not have constant holomorphic sectional curvature for the generic choice of parameters $\beta,c_1,c_2,d_1,d_2$. However, for some cases, parameters yielding constant holomorphic curvature metrics exist and may easily be computed.
\end{remark}

\begin{remark}
In Theorem \ref{thm:normalformwithv} the vector field $v$ need not be essential for the metric $g$, but only for its c-projective equivalence class $[g]=[g_0]$. In fact, $v$ is a Killing vector field for $g$ in the cases L1,CL1 and it is an infinitesimal homothety for $g$ in the cases L2,L3,CL2,CL3. In the cases L4,CL4, $v$ is indeed an essential vector field for $g$. In the degenerate case, $v$ is essential for the metric $g$ in the cases D1 and D3 and it is an infinitesimal homothety for the metric in D2.
\end{remark}

 Theorem \ref{thm:normalformwithv} describes all classes of c-projectively equivalent metrics admitting an essential c-projective vector field. In order to describe all metrics
   admitting an essential c-projective vector field we need to describe all metrics within these equivalence classes. The answer is given by:

\begin{thm}\label{thm:allmetricsintheclass}
Let $(g,J)$ be one of the local K\"ahler structures  of Theorem \ref{thm:normalformwithv}. Assume that $ g$ does not have constant holomorphic curvature and $\hat g\in [g]$.  Then $\hat g$ is either proportional to $ g$ with a constant coefficient  or given in the same coordinate system  by  the formulas below (provided the parameters $C$ and $c$ are such that $\hat g$ is well defined and nondegenerate).

Liouville case:
\begin{equation}\label{eq:ghat1}
\aligned
\hat{g}=&\frac{C}{(\rho-c)^2(\sigma-c)^2(\rho-\sigma)}\Bigg[(\rho-\sigma)^2(\rho-c)(\sigma-c)\left(\frac{F^2}{\rho-c}dx^2+\epsilon\frac{G^2}{\sigma-c}dy^2\right)+\\
&+\left(\left(\frac{\rho'}{F}\right)^2(\sigma-c)+\epsilon
(\rho-c) \left(\frac{\sigma'}{G}\right)^2\right)ds^2+\left(\left(\frac{\rho'}{F}\right)^2\sigma^2(\sigma-c)+\epsilon \rho^2(\rho-c) \left(\frac{\sigma'}{G}\right)^2\right)dt^2+\\
&+2\left(\left(\frac{\rho'}{F}\right)^2\sigma(\sigma-c)+\epsilon \rho(\rho-c) \left(\frac{\sigma'}{G}\right)^2\right)dsdt\Bigg]\endaligned
\end{equation}
\indent Complex Liouville case:
\begin{equation}\label{eq:ghat2}
\aligned
\hat{g}=&\frac{C}{(\rho-c)^2(\bar{\rho}-c)^2(\bar\rho-\rho)}\Bigg[\frac{1}{4}(\bar{\rho}-\rho)^2(\rho-c)(\bar\rho-c)\left(\frac{F^2}{\rho-c}dz^2-\frac{\bar{F}^2}{\bar{\rho}-c}d\bar{z}^2\right)+\\
&+\left(\left(\frac{1}{\bar{F}}\frac{\partial \bar{\rho}}{\partial \bar{z}}\right)^2 (\rho-c)-  \left(\frac{1}{F}\frac{\partial \rho}{\partial z}\right)^2(\bar{\rho}-c)\right)ds^2+\left(\left(\frac{1}{\bar{F}}\frac{\partial \bar{\rho}}{\partial \bar{z}}\right)^2 \rho^2(\rho-c)-\left(\frac{1}{F}\frac{\partial \rho}{\partial z}\right)^2\bar{\rho}^2(\bar{\rho}-c)\right)dt^2+\\
&+2\left(\left(\frac{1}{\bar{F}}\frac{\partial \bar{\rho}}{\partial \bar{z}}\right)^2 \rho(\rho-c)- \left(\frac{1}{F}\frac{\partial \rho}{\partial z}\right)^2\bar{\rho}(\bar{\rho}-c)\right)dsdt\Bigg].
\endaligned
\end{equation}
\indent Degenerate case:
\begin{align}
\hat{g}=\frac{C}{c(\rho-c)}\left(\frac{\rho}{c}h+\frac{\rho F^2}{\rho-c}dx^2+\frac{1}{\rho(\rho-c)}\left(\frac{\rho'}{F}\right)^2\theta^2\right).\label{eq:ghat3}
\end{align}
\end{thm}

Combining Theorems \ref{thm:normalformwithv} and \ref{thm:allmetricsintheclass}, we obtain a complete list of K\"ahler metrics of nonconstant holomorphic curvature admitting a c-projective vector field. As an application of the local description of essential c-projective vector fields given by Theorems \ref{thm:normalformwithv} and \ref{thm:allmetricsintheclass}, we obtain:

\begin{thm}\label{YOconjecture}
Let  $(M^4, g, J)$ be a closed K\"ahler surface of arbitrary signature admitting an essential c-projective vector field $v$. Then, for a certain constant $C\ne 0$ the
metric $C\cdot  g$ is a Riemannian metric of   constant positive
holomorphic curvature and therefore $(M^4, C \cdot g,  J)$  is isometric to
$\mathbb{CP}^2$ equipped with the Fubini-Study metric and the canonical complex structure.
\end{thm}

{ The  Riemannian analog of this statement is known as the Yano-Obata conjecture. It was stated and much studied in the 70th, see \cite[\S1.2]{MR2998672} for a historical overview. A proof valid in full generality (in the Riemannian case) appeared only recently  in  \cite{MR2998672}. We hope that the methods and ideas 
developed here will allow to treat the higher-dimensional case in arbitrary signature as well.}

\subsection{Main idea of the proof} \label{trick}

The existence of a c-projective vector field for  $(g, J)$ can easily be cast as an overdetermined \textsc{pde} system on the components of $g$, $J$ and $v$. The system is nonlinear, of second order in the derivatives of the unknown functions and is not tractable by standard methods.  Another method to study existence of c-projective vector fields is based on an observation of Mikes and Domashev  \cite{DomMik1978} who found a linear system of \textsc{pde}s whose solutions correspond to K\"ahler metrics that are c-projectively equivalent to a given one. The system of Mikes-Domashev allows one to rewrite the system of \textsc{pde}s that corresponds to the existence of a c-projective vector field as a linear system of \textsc{pde}s. The system is overdetermined, but unfortunately of third order.

Let us explain a trick that allows us to reduce the existence problem for c-projective vector fields to solving a \textsc{pde} system  with more equations on  less number of  unknown functions (\textsc{pde} systems of higher degree of over-determinacy are usually easier to solve) and which is of first order in the derivatives.  The occurrence of such a simplification might be surprising and indeed requires several preliminary results.  The trick was recently effectively used to solve the Lie problem \cite{MR2892455} and is explained in \cite[\S 2, especially \S 2.3]{MR2892455}. One of the two ideas behind this trick   already appeared in projective geometry in \cite{Fubini,Sol} and the other, in a certain form and in dimension two, in \cite{liouville} (see also~\cite{MR2384718} for the higher dimensional case).
 The trick was already used in c-projective geometry in the proof of the {  (Riemannian)} Yano-Obata conjecture  \cite{MR2998672}.   Let us explain the rough schema/idea/tools of the trick,  the details and the calculations are in \S \ref{sec:solvinglie}.

Let $(M, J)$ be a complex manifold of real dimension $2n\geq 4$. We  define a \textit{c-projective structure} $([\nabla],J)$ on $M$ as  an equivalence class of $J$-parallel affine torsion-free connections $\nabla$ on $TM$. Two such connections are (c-projectively) equivalent if they have the same $J$-planar curves. The condition that two  connections   are c-metrisable equivalent   is in fact an easy linear algebraic condition,  see \eqref{eq:christtrafo}. Certain c-projective structures $([\nabla],J)$ contain a Levi-Civita
connection of a metric that is K\"ahler with respect to $J$. In this case we say that the metric is \textit{compatible} with the c-projective structure,
and the c-projective structure is \emph{metrisable}.

It was recently observed in \cite{MR2998672} that metrics which are compatible with a c-projective structure $([\nabla],J)$ are in one-to-one correspondence with nondegenerate solutions of a certain overdetermined system of linear partial differential equation. Actually, as we already mentioned above, the existence of such a linear system of \textsc{pde}s was
known before, see \cite{DomMik1978}   or  \cite[Chapter 5, \S 2]{sinjukov}. The advantage of the  modification of this system suggested in  \cite{MR2998672} is that the obtained system is c-projectively invariant.  In the language of  Cartan geometry, the system was obtained and explained  in  \cite{mettlerkahmetri}.

Since the system is linear, its space of solutions is a vector space. If its dimension is one, then all metrics compatible with the c-projective structure are mutually proportional with a constant coefficient of proportionality. Then, every c-projective vector field is a homothety or a Killing vector field.  Hence, it is not essential and therefore not of interest for this article. Now, as it follows from \cite[Lemma 6]{MR2948791} (and also from the earlier paper
~\cite{MR2144249} where a similar statement was proven in the language of Hamiltonian $2$-forms for positively definite metrics), if the space of solutions is at least three-dimensional, the metric has constant holomorphic sectional curvature and we are done. Thus, the interesting case is when there exist two solutions $h, \hat h$ of this system and any other solution is a linear combination of these two. 

We consider the Lie-derivative of these solutions w.r.t. the c-projective vector field $v$. Since the solutions are sections of an associated tensor bundle, their Lie derivative is a well-defined section of the same bundle and by standard arguments (using that the system is c-projectively invariant) one concludes that it is also a solution. Since the space of solutions is two-dimensional, we have
\begin{equation} \label{big number}
\mathcal{L}_v h = \gamma h + \delta \hat h  \quad \textrm{and} \quad \mathcal{L}_v \hat h = \alpha h + \beta \hat h
\end{equation}
for certain constants $\alpha$,..., $\delta$. In other words, $\mathcal{L}_v$ is an endomorphism of the (two-dimensional)
space  of solutions of the linear \textsc{pde} system we are interested in. If we replace $h$ and $\hat h$ by another pair of linearly independent solutions, i.e., if we consider another basis in the space of solutions, then the matrix $\begin{pmatrix} \gamma & \alpha \\ \delta & \beta\end{pmatrix}$ changes by a similarity transformation. Moreover, if we scale the projective vector field by a  constant, the matrix of the corresponding endomorphism will be scaled by the same constant. Thus, without loss of generality we may assume that  the matrix $\begin{pmatrix} \gamma & \alpha \\ \delta & \beta\end{pmatrix}$ is given by \eqref{eq:normalformsv}.

Now, the local description of the pairs of c-projectively equivalent metrics gives us the form of the solutions $h$, $\hat h$ in a local coordinate system. The local form depends on two functions of one variable, or on one holomorphic function of one variable, or on one function of one variable and one function of two variables which we consider to be unknown functions. In addition, the four components of the projective vector field are also considered to be unknown functions. The number of equations is relatively big:  each equation of
 \eqref{big number} consists of actually six equations so altogether we have twelve equations. It appears that it is possible and relatively easy to solve this system, which we do.  The result is the desired local classification of c-projective vector fields.

The proof of the Yano-Obata conjecture (Theorem \ref{YOconjecture}) consists of two steps. First we prove that c-projectively equivalent metrics $g$ and $\hat g$ of complex Liouville type  (see Theorems \ref{thm:normalformwithv} and \ref{thm:allmetricsintheclass})  cannot exist on a compact K\"ahler surface (unless they are affinely equivalent).  The point is that in the complex Liouville case, the pair of metrics $g$ and $\hat g$ gives rise to a natural $(1,1)$-tensor field $A=A(g,\hat g)$ with a complex eigenvalue $\rho$  that behaves as a holomorphic function on an open and (locally) dense set.  As $M$ is compact, this property finally leads to a contradiction with the maximum principle {unless $\rho$ is a constant} which implies that  the metrics are affinely equivalent.

In the other two (Liouville and degenerate) cases  we analyse the behaviour of the eigenvalues of $A=A(g,\bar g)$ and the scalar curvature of $g$ along trajectories of a c-projective vector field $v$. It turns out that the eigenvalues of $A$ may be bounded only in the subcases L2 and D2 from Theorem \ref{thm:normalformwithv}. But in these two subcases, the explicit computation of the scalar curvature $\mathrm{Scal} (\tau(t))$ along a generic integral curve $\tau(t)$ of $v$  shows that  the boundedness of   $\mathrm{Scal} (\tau(t))$ amounts to the fact that  $g$ has constant holomorphic curvature.

\subsection{Outline} \label{sec:background}

In \S2 we recall and collect previously obtained results about c-projectively equivalent metrics and c-projective vector fields. We give precise references whenever possible, and prove most of the results for self-containedness (since certain results were obtained in a different mathematical setup, e.g., in the language of Hamiltonian $2$-forms, and it is easier to prove these statements than to translate them).

In \S3 we partially extend the classification of Apostolov et al. to the four-dimensional (pseudo-Riemannian) K\"ahler case, see Theorem \ref{thm:normalform}.  

Having the classification of Theorem \ref{thm:normalform} at hand, in \S4 we apply the technique explained in  \S\ref{trick} to reduce our problem to a first order \textsc{pde}-system  on $v$ and $g$ in Frobenius form {(i.e., such that all first  derivatives of the unknown functions are  explicit expressions in the unknown functions)}, and solve it.  

Finally, \S5 contains the proof of the Yano-Obata conjecture.

\section{Preliminaries}\label{sec:prelim}

\subsection{C-projective structures and compatible K\"ahler metrics}\label{sec:compks}

Let $M$ be a complex manifold of real dimension $2n\geq 4$ with complex structure $J$. Tashiro showed~\cite{MR0087181} that two $J$-linear torsion-free connections $\nabla$ and $\hat{\nabla}$ are c-projectively equivalent if and only if there exists a $1$-form $\Phi$ such that
\begin{align}
\hat{\nabla}_X Y-\nabla_X Y=\Phi(X)Y+\Phi(Y)X-\Phi(JX)JY-\Phi(JY)JX\label{eq:christtrafo}
\end{align}
for all vector fields $X,Y \in \Gamma(TM)$. The equivalence class of connections c-projectively equivalent to $\nabla$ will be denoted by $[\nabla]$. A K\"ahler metric $g$ on $M$ is said to be compatible with $([\nabla],J)$ if  its Levi-Civita connection $\gnabla$ is an element of $[\nabla]$.  Clearly, two K\"ahler metrics are c-projectively equivalent if and only if they are compatible with the same c-projective structure.

Let $\mathcal{E}(\tfrac{1}{n+1})=(\Lambda^{2n} T^* M)^{\frac{1}{n+1}}$ denote the bundle of volume forms of c-projective weight $\tfrac{1}{n+1}$. By definition, its transition functions are those of $\Lambda^{2n} T^* M$ taken to the power $\tfrac{1}{n+1}$. Let $S^2_J T^* M$ be the bundle of hermitian symmetric $(2,0)$ tensors and denote by $S^2_J T^* M (\tfrac{1}{n+1})=S^2_J T^* M \otimes \mathcal{E}(\tfrac{1}{n+1})$ its weighted version. For any $J$-linear torsion-free connection $\nabla$, consider the linear \textsc{pde} system
\begin{align}
\nabla_X h=X\odot \Lambda+JX\odot J\Lambda
\label{eq:maininvariant}
\end{align}
on sections $h$ of $S^2_J T^* M(\tfrac{1}{n+1})$, where $X\odot Y=X\otimes Y+Y\otimes X$ is the symmetric tensor product and $\Lambda^i=\frac{1}{2n}\nabla_k h^{ki}$ is a vector field of weight $\tfrac{1}{n+1}$.

In~\cite{MR2998672}, it was shown that equation \eqref{eq:maininvariant} does not change if we replace $\nabla$ in \eqref{eq:maininvariant} by another connection in $[\nabla]$, in other words, \eqref{eq:maininvariant} is c-projectively invariant. Furthermore, it was shown that a K\"ahler metric $g$ on $(M,J)$ is compatible with $([\nabla],J)$ if and only if the $(2,0)$ tensor field
\begin{align}
h_g=g^{-1}\otimes (\det g)^{1/(2n+2)}\in \Gamma (S^2_J T^* M (\tfrac{1}{n+1}))\label{eq:metricsolution}
\end{align}
is a solution of \eqref{eq:maininvariant}. Conversely, every nondegenerate section $h$ of $S^2_J T^* M(\tfrac{1}{n+1})$ solving \eqref{eq:maininvariant} gives rise to a unique K\"ahler metric $g$ compatible with $([\nabla],J)$.

Equation \eqref{eq:maininvariant} is linear and of finite type. Consequently,  the space of its solutions $\mathcal{S}([\nabla],J)$ is a finite-dimensional vector space whose dimension $d([\nabla],J)$ is called the \textit{degree of mobility of} the c-projective structure $([\nabla],J)$. The generic c-projective structure has degree of mobility $0$  and it remains an open problem to characterise the c-projective structures having degree of mobility at least one (see~\cite{mettlerkahmetri} for partial results in the surface case). For a K\"ahler metric $g$ on $(M,J)$ we define $d(g,J)=d([\gnabla],J)$ and call it the \textit{degree of mobility} of the K\"ahler structure $(g,J)$. {Note that we always have $d(g,J)\geq 1$ since $h_g$ defined in \eqref{eq:metricsolution} solves \eqref{eq:maininvariant}.}

If the c-projective structure $([\nabla],J)$ admits a compatible K\"ahler metric $g$, the equation \eqref{eq:metricsolution} is more or less equivalent to the one obtained by
Mikes and Domashev \cite{DomMik1978}. They  showed that the K\"ahler metrics $\hat{g}$ on $(M,J)$  which are c-projectively equivalent to $g$, are in one-to-one correspondence with nondegenerate  $g$-symmetric endomorphisms $A:TM\rightarrow TM$ commuting with $J$  and solving
\begin{align}
\gnabla_X A=X^{\flat}\otimes \Lambda+\Lambda^\flat \otimes X+(JX)^{\flat}\otimes J\Lambda+(J\Lambda)^\flat\otimes JX\label{eq:main}
\end{align}
for every vector field $X \in \Gamma(TM)$, where $\Lambda=\tfrac{1}{4}\grad_g(\tr\,A)$ and $X^\flat=g(X,.)$. The correspondence is given by
\begin{equation}
A=A(g,\hat{g})=\left(\frac{\det\hat{g}}{\det g}\right)^{1/(2(n+1))}{\hat{g}^{-1}g,}\label{eq:defA}
\end{equation}
{where  $g,\hat g:TM\rightarrow T^* M$ are viewed as bundle isomorphisms and $\hat g^{-1}g:TM\rightarrow TM$ denotes the composition. In local coordinates, we have $(\hat g^{-1}g)^i_j=\hat g^{ik}g_{kj}$, where 
$\hat g^{ik}\hat g_{kj}=\delta^i_j$, so the matrix of  the endomorphism 
$\hat g^{-1}g$ is the product of  the inverse to the matrix of $\hat g$ and of the matrix of $g$. 
We will use similar notation throughout the paper. If for instance $A:TM\rightarrow TM$ is a tensor of type $(1,1)$, then the composition $gA:TM\rightarrow T^* M$ is the tensor of type $(0,2)$ given in local coordinates by $(gA)_{ij}=g_{ik}A^k_j$.}

The linear space of $g$-symmetric $J$-commuting solutions of \eqref{eq:main} will be denoted by $\mathcal{A}(g,J)$. It is easy to check that this space is isomorphic to $\mathcal{S}([\nabla],J)$ via the map $\varphi_g : \mathcal{S}([\nabla],J) \to \mathcal{A}(g,J)$ defined by
\begin{equation}\label{eq:solspaceiso}
\varphi_g(h)=hh_{g}^{-1}
\end{equation}
where $h_g$ is defined by \eqref{eq:metricsolution}. In particular, we have $\varphi_g(h_g)=\mathrm{Id}$. 

As we will explain below, the existence of a c-projective vector field $v$ for $(g,J)$ gives rise to a c-projectively equivalent metric $\hat{g}$. A first step towards the classification of local four-dimensional K\"ahler structures admitting a c-projective vector field is thus to classify the local four-dimensional K\"ahler structures for which equation \eqref{eq:main} admits a (non-trivial) solution. Since we are looking for essential c-projective vector fields, we are seeking for K\"ahler metrics admitting an $A\in \mathcal{A}(g,J)$ which is non-parallel.

In the Riemannian case the classification of K\"ahler structures admitting solutions of \eqref{eq:main} is already known. Indeed, as the reader may easily verify, the elements $A\in \mathcal{A}(g,J)$ are in one-to-one correspondence with the real $(1,\! 1)$-forms $\phi$ on $(M,g,J)$ satisfying
$$
\gnabla_X\phi=\frac{1}{2}\left(d \tr_{\omega}\phi\wedge(JX)^{\sharp}-Jd \tr_{\omega}\phi\wedge X^{\sharp}\right)
$$
for every vector field $X\in \Gamma(TM)$. The correspondence is given by $\phi=g(AJ\cdot,\cdot)$. The solutions $\phi$ of the above equation are called \textit{Hamiltonian} $2$-\textit{forms}. Inspired by the work of Bryant~\cite{MR1824987}, Apostolov et al.~\cite{MR2228318} obtained a complete local classification of Hamiltonian $2$-forms (in all dimensions) and subsequently developed a comprehensive global theory~\cite{MR2144249} with applications in extremal~\cite{MR2425136} - and weakly Bochner-flat K\"ahler metrics~\cite{MR2411469}.

Of course, the definition of a c-projective vector field does not require a metric: let  $([\nabla],J)$ be a c-projective structure on the complex manifold $(M,J)$ of real dimension $2n\geq 4$. A vector field $v$ on $M$ is said to be c-projective with respect to $([\nabla],J)$ if its (locally defined) flow preserves the $J$-planar curves of $([\nabla],J)$. The set of c-projective vector fields on $M$ with respect to a given c-projective structure $([\nabla],J)$ is a Lie algebra and will be denoted by $\mathfrak{p}([\nabla],J)$. For a K\"ahler metric $g$ on $(M,J)$, we define $\mathfrak{p}(g,J)=\mathfrak{p}([\gnabla],J)$.

As we already mentioned, K\"ahler structures $(g,J)$ admitting an essential c-projective vector field $v$ necessarily  have degree of mobility $d(g,J)\geq 2$. Indeed, let $v$ be a c-projective vector field. Then, in a neighborhood of every point we can define a metric $g_t=(\phi_{t}^v)^* g$ for small values of $t$, that by assumption is c-projectively equivalent to $g$ on this neighborhood. Then, $A_{t}=A(g,g_{t})$ constructed by \eqref{eq:defA} is contained in $\mathcal{A}(g,J)$. Note that the derivative of the tensor $A_t$ at $t=0$ is an element of $\mathcal{A}(g,J)$. Calculating this derivative explicitly using \eqref{eq:defA}, we obtain that a vector field $v$ on $(M,g,J)$ whose flow preserves $J$ is c-projective with respect to the K\"ahler metric $g$ if and only if the symmetric $J$-linear endomorphism  {$A_v$  defined by}
\begin{equation}\label{hprojvf}
A_v=\left.\frac{d A_t}{d t}\right|_{t=0}= g^{-1}\mathcal{L}_v g -\frac{1}{2(n+1)}\tr (g^{-1}\mathcal{L}_v g)\Id,
\end{equation}
is contained in $\mathcal{A}(g,J)$. It is straightforward to check that $v$ is an infinitesimal homothety if and only if $A_v$ is proportional to $\mathrm{Id}$. Thus, if $v$ is essential, $\mathrm{Id}$ and $A_v$ are linearly independent which implies that $d(g,J)\geq 2$.

Many proofs in our paper essentially use some algebraic properties of the tensors $g$, $J$ and $A$ which can be easily deduced from the following simultaneous canonical form for $g$, $J$ and $A$.  Let $g$ be a symmetric nondegenerate bilinear form on a vector space $V$ and $J: V\to V$ a complex structure such that $g(J\cdot , \cdot)=-g(\cdot , J\cdot)$.
Suppose that an operator $A:V\to  V$ is $g$-symmetric and commutes with $J$. Then, in an appropriate basis, the matrices of $g$, $J$ and $A$  can be simultaneously reduced to the following forms:
\begin{equation}
\label{eq:propA1}
A= \begin{pmatrix} A' & 0 \\ 0 & A' \end{pmatrix}, \quad  g = \begin{pmatrix} g' & 0 \\ 0 & g' \end{pmatrix} \quad \mbox{and} \quad  J = \begin{pmatrix} 0 & \mathrm{Id} \\  -  \mathrm{Id} & 0 \end{pmatrix}
\end{equation}
where $g' A' =  (A')^\top g'$, i.e., $A'$ is $g'$-symmetric.  In particular, the determinant of $A$ is a full square and below we use the notation $\sqrt{\mathrm{det}A}$  for $\det A'$. Notice that  $\sqrt{\mathrm{det}A}$ is necessarily real but might be negative.

Moreover in dimension 4, there are three essentially different cases for $A'$ and $g'$:
\begin{equation}
\label{eq:propA2}
\begin{array}{lll}
A' = \begin{pmatrix}   \rho & 0 \\ 0 & \sigma  \end{pmatrix}, & g' = \begin{pmatrix} \epsilon_1 & 0 \\ 0 & \epsilon_2 \end{pmatrix},  & \rho,\sigma\in \R \mbox{ and } \epsilon_i=\pm 1;
\\  & & \\
A' = \begin{pmatrix}   \mathcal R & \mathcal I \\ -\mathcal I & \mathcal R  \end{pmatrix}, & g' = \begin{pmatrix} 0 & 1 \\ 1 &  0  \end{pmatrix},  & \rho=\mathcal R + \mathrm{i} \mathcal I \in \mathbb C,  \ \mathcal I \ne 0;
\\ & & \\
A' = \begin{pmatrix}   \rho &  1 \\  0 &  \rho  \end{pmatrix}, & g' = \begin{pmatrix} 0 & 1 \\ 1 &  0  \end{pmatrix},  & \rho\in \R.
\end{array}
\end{equation}

\subsection{Hamiltonian Killing vector fields}\label{sec:killing}

An important tool in the theory of c-projectively equivalent K\"ahler metrics are certain Killing vector fields for $g$ that can be canonically constructed from $g$ and $A$. The Killing fields are Hamiltonian (i.e., they are skew-gradient for certain functions on $M$) with respect to the symplectic structure $g(J.,.)$ on $M$. In the Riemannian situation, the results of this section are special cases of the results in~\cite{MR2228318}, see also \cite{Kiyohara2010}.

\begin{lem}\label{lem:killing}
Let $(M,g,J,\omega)$ be a K\"ahler manifold {of real dimension $2n\geq 4$} and let $A\in \mathcal{A}(g,J)$ be non-degenerate. Then the $\omega$-Hamiltonian vector field $X_H$ of the Hamiltonian $H=\sqrt{\det A}$ is Killing.
\end{lem}
\begin{proof}
Let $\hat{g}$ be a K\"ahler metric on $(M,J)$ which is c-projectively equivalent to $g$ so that $A=A(g,\hat{g})$. Recall from~\eqref{eq:christtrafo} that the Levi-Civita connections $\gnabla,{}^{\hat{g}}\nabla$ are related by
$$
{}^{\hat{g}}\nabla_X Y-\gnabla_X Y=\Phi(X)Y+\Phi(Y)X-\Phi(JX)JY-\Phi(JY)JX
$$
for some $1$-form $\Phi$ on $M$. Contracting this last equation implies $\Phi=d\phi$, where
$$\phi=\frac{1}{4(n+1)}\ln\left(\frac{\det\,\hat{g}}{\det\,g}\right)$$
and hence
$$\sqrt{\det(A)}=\left(\frac{\det\,\hat{g}}{\det\,g}\right)^{-\frac{1}{2(n+1)}}=\mathrm{e}^{-2\phi}.$$
Consequently, in order to show that $\sqrt{\det A}$ is the Hamiltonian for a Killing vector field, it suffices to show that the $(0,2)$-tensor
$$
2\Phi\otimes \Phi-\gnabla\Phi
$$
is hermitian. Equation \eqref{eq:christtrafo} and straightforward calculations yield
$$
\mathrm{Ric}(\hat{g})-\mathrm{Ric}(g)+2(n+1)(\Phi \otimes \Phi+\Phi(J\cdot)\otimes\Phi(J\cdot))=-2(n+1)(\gnabla\Phi-2\Phi\otimes \Phi)
$$
where $\mathrm{Ric}(g)$ and $\mathrm{Ric}(\hat{g})$ denote the hermitian Ricci tensors of $g$ and $\hat{g}$ respectively. Since the left hand side of the last equation is hermitian, the claim follows.
\end{proof}

Replacing $A$ with $A-t\,\Id$ in Lemma \ref{lem:killing}, we can expand the expression $\sqrt{\mathrm{det}(A-t\,\Id)}$ at every point $p$ of $M$ as a polynomial of degree $n$ in $t$ (recall that $\sqrt{\det A}=\det A'$ and similarly $\sqrt{\det(A-t\, \Id)}=\det{(A' - t\, \Id')}$).  Hence,
\begin{align}
\sqrt{\mathrm{det}(A-t\,\Id)}=(-1)^n t^n+(-1)^{n-1}\mu_1(p) t^{n-1}+...+\mu_n(p),\label{eq:defsigmai}
\end{align}
where $\mu_i(p)$ are the elementary symmetric functions in the (possibly complex) eigenvalues $$\lambda_1(p),...,\lambda_n(p)$$ of the $J$-linear $g$-symmetric endomorphism $A:T_p M \rightarrow T_p M$. Note that in view of \eqref{eq:propA1} each eigenvalue of $A$ has even multiplicity.

From Lemma \ref{lem:killing}, we immediately obtain

\begin{cor}\label{cor:killing}
Let $(M,g,J,\omega=g(J.,.))$ be a K\"ahler $n$-manifold and  $A\in \mathcal{A}(g,J)$. Let $\mu_i$ be the functions defined in \eqref{eq:defsigmai}, $i=1,...,n$. Then the corresponding $\omega$-hamiltonian vector fields $K_i=X_{\mu_i}$ are Killing.
\end{cor}

\begin{remark}
We see that the vector field $J\mathrm{grad}_g (\tr\,A)$ in equation \eqref{eq:main} coincides with $2K_1$. The fact that this vector field is Killing is well-known in the classical c-projective literature.
\end{remark}

Since we will need it frequently throughout the article we give a proof of the next simple statement:

\begin{lem}\label{lem:easy}
Let $(M,g,J,\omega=g(J.,.))$ be a connected K\"ahler manifold and let $f$ be a smooth function on $M$ such that $f$ and $f^2$ are hamiltonian functions for Killing vector fields on $M$. Then, $f$ is a constant.
\end{lem}

\begin{proof}
Let $f$ and $f^2$ be hamiltonian functions for Killing vector fields, i.e., the functions
$$
df(J\dot{\gamma}(t))\quad \mbox{and} \quad f(\gamma(t)) df(J\dot{\gamma}(t))
$$
are constant along every geodesic $\gamma$ of $g$, that is, they do not depend on $t$. Then, $f$ is constant on every convex neighborhood and consequently, $f$ is constant on $M$.
\end{proof}

We next derive some properties of the eigenvalues of $A$.

\begin{lem}\label{lem:eigenvaluestructureA}
Let $(M,g,J)$ be a connected K\"ahler surface and let $A\in \mathcal{A}(g,J)$ be non-parallel. Then, there are continuous functions $\rho,\sigma:M\rightarrow \mathbb{C}$ and a decomposition
$$
M=M^r\sqcup M^{\mbox{\tiny sing}}\sqcup M^{c}
$$
of $M$ into disjoint sets, where $M^r,M^{c}$ are open and $M^{\mbox{\tiny sing}}$ is closed, such that $\rho(p)<\sigma(p)$ are real eigenvalues of $A$ for all $p\in M^r$, $\rho(q)=\overline{\sigma(q)}$ are complex-conjugate eigenvalues for $A$ with non-zero imaginary part for all $q\in M^{c}$ and in the points of $M^{\mbox{\tiny sing}}$, $A$ has s single real eigenvalue $\rho=\sigma$ of multiplicity $4$.

Moreover, the subset $M'=M^r\cup M^{c}$ is dense in $M$ and $\rho,\sigma$ are smooth on $M'$. 
\end{lem}

\begin{proof}
Recall that at every point $p$ of $M$ the eigenvalues $\rho(p),\sigma(p)$ of $A:T_p M\rightarrow T_p M$ are solutions to the quadratic equation
$$
t^2-\mu_1(p) \, t +\mu_2(p) =0,
$$
where $\mu_1,\mu_2$ are defined in \eqref{eq:defsigmai}. Defining $f=\mu_1^2-4\mu_2$, we have that $M^r=f^{-1}((0,\infty))$ is the set of points $p$ where $A$ has two different real eigenvalues $\rho(p)\neq\sigma(p)$ and $M^{c}=f^{-1}((-\infty,0))$ is the set of points $q$ where $A$ has complex-conjugate eigenvalues $\rho(q)=\overline{\sigma(q)}$ with non-vanishing imaginary parts. Thus, $M$ is the disjoint union
$$
M=M^{r}\sqcup f^{-1}(0)\sqcup M^{c}
$$
of the open subsets $M^r,M^{c}$ and the closed subset $M^{sing}=f^{-1}(0)$ of points where $A$ has a single real eigenvalue $\rho$ of multiplicity $4$. Moreover, if $M^{sing
}$ contains an open subset $U$, we have that by Corollary \ref{cor:killing}, $\rho$ and $\rho^2$ are hamiltonian functions for Killing vector fields on $U$ and from Lemma \ref{lem:easy} it follows that $\rho$ is a constant. Then, the Killing vector field $J\mathrm{grad}_g(\tr\,A)$ vanishes on $U$ and hence on $M$, implying that $A$ is parallel on $M$. Consequently, excluding the case that $A$ is parallel, the set $M^{\mbox{\tiny sing}}$ does not contain an open subset, hence, $M'=M^r\cup M^{c}$ is dense in $M$.

On $M^r$, we can define functions $\rho,\sigma:M^r\rightarrow \mathbb{R}$ by the ordering $\rho(p)<\sigma(p)$ of the eigenvalues for all $p\in M^r$. On $M^{c}$ we define complex conjugate functions $\rho,\sigma:M^{c}\rightarrow \mathbb{C}$ by assuming that the imaginary part of the eigenvalue $\rho(q)$ is smaller than that of $\sigma(q)$ for all $q\in M^{c}$. The functions so defined can be extended to give continuous functions $\rho,\sigma:M\rightarrow \mathbb{C}$ which are smooth on $M'$.
\end{proof}

Of course, some of the sets $M^r$, $M^{c}$ or $M^{sing}$ might be empty. Note that Lemma \ref{lem:eigenvaluestructureA} also contains the case when on $M^r$, the endomorphism $A$ has a non-constant eigenvalue, say $\rho$, and a constant eigenvalue $\sigma=c$.

Notice that this lemma automatically excludes the case of a $2\times 2$ Jordan block (see \eqref{eq:propA1} and \eqref{eq:propA2}) from our further consideration. Indeed, if $A$ is conjugate to $\begin{pmatrix} A' & 0 \\ 0 & A'\end{pmatrix}$ with $A' = \begin{pmatrix} \rho  & 1 \\ 0  &\rho\end{pmatrix}$ on an open non-empty subset $U$, then $U\subset M^{sing}$ and, as we have just shown, $A$ is parallel.

Theorem \ref{thm:normalformwithv} we are going to prove gives a local description of c-projective vector fields in a neighbourhood of {\it almost every point}.  Now we are able to characterise such points explicitly in terms of the eigenvalues of $A$.  Consider the following subset
$$
M^\circ= \{ p\in M~|~  \rho(p)\ne \sigma(p) \ \mbox{ and } \ d\rho(p)\ne 0, \ d\sigma(p)\ne 0  \} \subseteq  M'
$$
or, if one of the eigenvalues, say $\sigma$, is constant on the whole of $M$:
$$
M^\circ = \{ p\in M~|~  \rho(p)\ne \sigma \ \mbox{ and } d\rho(p)\ne 0  \} \subseteq M'.
$$

\begin{lem}\label{lem:propM0}
 $M^\circ$ is open and dense in $M$.
\end{lem}

\begin{proof} Clearly, $M^\circ$ is open. To prove that $M^\circ$ is dense
assume, by contradiction, that the differential $d\rho$ of the non-constant eigenvalue $\rho$ vanishes on some open subset $U\subseteq M'$. Suppose first that $U\subseteq M^r$. Then, $\rho$ is equal to a real constant $c$ on $U$. Let $K_t$ be the Killing vector field corresponding to the hamiltonian function $\mu_t=t^2-\mu_1t+\mu_2$, see \eqref{eq:defsigmai}. Then, $K_c$ vanishes on $U$ since $\mu_c$ vanishes on $U$ and therefore, $K_c$ vanishes on the whole $M$ implying that $\rho=c$ on $M$ contradicting the assumption that $\rho$ is non-constant. Similarly, if $U\subseteq M^{c}$, we have that $\rho$ and hence $\sigma=\bar{\rho}$ are equal to complex-conjugate constants on $U$. Then, $A$ is parallel on $U$, hence parallel on $M$ and therefore $\rho$ is constant on the whole $M$ contradicting the assumptions.
\end{proof}

\begin{lem}\label{lem:eigenvaluegradients}
Let $(M,g,J)$ be a connected K\"ahler surface and let $A\in \mathcal{A}(g,J)$. Then the gradients of the eigenvalues $\rho,\sigma$ of $A$ on $M'$ are contained in the eigenspaces of $A$ corresponding to $\rho,\sigma$ respectively. In particular, the gradients of the (non-constant) eigenvalues  $\rho,\sigma$ are linearly independent at each point $p\in M^\circ$.
\end{lem}

\begin{proof}
Let $Y\in \Gamma(T^{\mathbb{C}}M)$  {(where $T^{\C}M=TM\otimes_{\R}\C$)} be a smooth complex vector field such that $AY=\sigma Y$. Taking the covariant derivative of the equation $AY=\sigma Y$ with respect to a vector field $X\in \Gamma(T^{\mathbb{C}}M)$ and inserting \eqref{eq:main}, we obtain
\begin{align}
(A-\sigma\Id)\nabla_X Y=X(\sigma)Y-g(Y,X)\Lambda-g(Y,\Lambda)X-g(Y,JX)J\Lambda-g(Y,J\Lambda)JX,\label{eq:covderiveigvec}
\end{align}
where $\Lambda=\tfrac{1}{4}\mathrm{grad}_g (\tr\,A)$ and all operations are extended complex-linearly from $TM$ to $T^{\mathbb{C}}M$. Now insert a vector field $X$ such that $AX=\rho X$ into this equation, where we assume that $\rho\neq \sigma$. Using $g(X,Y)=g(JX,Y)=0$, we obtain
$$(A-\sigma\Id)\nabla_X Y=X(\sigma)Y-g(Y,\Lambda)X-g(Y,J\Lambda)JX.$$
We can choose $Y$ in such a way that it is not a null vector, i.e. $g(Y,Y)\neq 0$. Inserting $Y$ together with the last equation into the metric, we obtain $0=X(\sigma)=g(\mathrm{grad}_g \sigma,X)$. This proves the lemma.
\end{proof}

\begin{cor}\label{cor:commutekilling}
Let $(M,g,J,\omega=g(J.,.))$ be a K\"ahler surface and let $A\in \mathcal{A}(g,J)$. Let $\mu_1=\rho+\sigma$ and $\mu_2=\rho\sigma$ and  $K_1=X_{\mu_1},K_2=X_{\mu_2}$ be the Killing vector fields from Corollary \ref{cor:killing} and define $V_1=-JK_1,V_2=-JK_2$. Then, we have the following:
\begin{enumerate}
\item The distributions $D=\mathrm{span}\{V_1,V_2\}$ and $JD=\mathrm{span}\{K_1,K_2\}$ defined on $M^\circ$ are orthogonal to each other and the restriction of the metric $g$ to each of these distributions is non-degenerate.
\item The vector fields $V_1,V_2,K_1,K_2$ are mutually commuting.
\item The leaves of the integrable distribution $D$ are totally geodesic.
\end{enumerate}
\end{cor}
\begin{proof}
When proving this lemma, we take into account the algebraic representations \eqref{eq:propA1} and \eqref{eq:propA2}. 

$(1)$ By definition $D$ is spanned by $V_1=\mathrm{grad}_g (\rho)+\mathrm{grad}_g(\sigma)$ and $V_2=\rho\,\mathrm{grad}_g (\sigma)+\sigma\, \mathrm{grad}_g(\rho)$. From Lemma \ref{lem:eigenvaluegradients}, we immediately obtain $g(V_1,JV_2)=0$.

Since $D$ and $JD$ are orthogonal, the restriction of $g$ to these distributions necessarily is non-degenerate.

$(2)$ It follows immediately from the first part, that $\omega(K_1,K_2)=0$. Hence, the functions $\mu_i$ Poisson commute, i.e. $\{\mu_1,\mu_2\}=0$, and thus $$[K_1,K_2]=[X_{\mu_1},X_{\mu_2}]=X_{\{\mu_1,\mu_2\}}=0.$$
Since the $K_i$ are hamiltonian Killing vector fields, they are holomorphic which implies that also the $V_i$ are holomorphic. Then, the vector fields $K_i,V_j$ mutually commute.

$(3)$ It follows from (2) that the distributions $D$ and $JD$ are integrable. If $A$ has two non-constant eigenvalues we have $TM=D\perp JD$. Then, since $JD$ is spanned by Killing vector fields, the leaves of the distribution $D$ are totally geodesic. If $A$ has a non-constant eigenvalue $\rho$ and a constant eigenvalue $c$, we still have that the one-dimensional leaves of $D$ are totally geodesic. Indeed, since $D$ is spanned by $V=\mathrm{grad}_g\,\rho$, we have to show that $\nabla_V V$ is proportional to $V$. Since $K=JV$ is Killing, we obtain $g(\nabla_V V,K)=-g( V,\nabla_V K)=0,$ thus $\nabla_V V$ has no components in the direction of $K$. Let $X$ be an eigenvector field corresponding to the constant eigenvalue $c$. Using Lemma \ref{lem:eigenvaluegradients} and equation \eqref{eq:covderiveigvec} from its proof, we see that $\nabla_V X$ is contained in the eigenspace of $A$ corresponding to the eigenvalue $c$. Then, $g(\nabla_V V,X)=-g( V,\nabla_V X)=0$ and the claim follows.
\end{proof}

\section{Normal forms for $4$-dimensional  K\"ahler structures admitting solutions to \eqref{eq:main}}\label{sec:localclass}

Our first goal is to  classify the four-dimensional local K\"ahler structures satisfying $d(g,J)>1$ and then look among the obtained K\"ahler structures for those admitting a c-projective vector field.

Theorem \ref{thm:normal form} below provides the desired classification. The new part is the complex Liouville case which occurs in split-signature only. The Riemannian part is a special case of the classification of Hamiltonian $2$-forms~\cite{MR2228318}. Besides a proof for the complex Liouville case we will also provide an alternative proof for the Liouville case. The degenerate case cannot be proved with our method but is contained in the statement for the sake of completeness.

\begin{thm}\label{thm:normalform}
Let $(M,g,J,\omega)$ be a K\"ahler surface. Suppose that $A\in \mathcal{A}(g,J)$ is non-parallel. Then, in a neighborhood of almost every point, we have one of the following cases:

\begin{itemize}

\item Liouville case: there are coordinates $x,y,s,t$ and functions $\rho(x),\sigma(y)$ such that
\begin{align}
\begin{array}{c}
g=(\rho-\sigma)(dx^2+\epsilon dy^2)+\frac{1}{\rho-\sigma}\left[(\rho')^2(ds+\sigma dt)^2+ \epsilon(\sigma')^2(ds+\rho dt)^2\right],\vspace{1mm}\\
\omega=\rho'dx\wedge (ds+\sigma dt)+\sigma'dy\wedge ( ds+\rho dt)
\weg{J\partial_x=\frac{1}{\rho'}(\rho \partial_s-\partial_t),\,\,\,J\partial_y=\frac{\epsilon}{\sigma'}(-\sigma\partial_s+\partial_t),\vspace{1mm}\\
J\partial_s=-\frac{1}{\rho-\sigma}(\rho'\partial_x+\epsilon\sigma'\partial_y),\,\,\,J\partial_t=-\frac{1}{\rho-\sigma}(\rho'\sigma\partial_x+\epsilon\rho\sigma'\partial_y),}
\end{array}\label{eq:normalform1}
\end{align}
and
\begin{align}
A=\rho \partial_x\otimes dx+\sigma\partial_y \otimes dy+(\rho+\sigma)\partial_s\otimes ds+\rho\sigma\partial_s\otimes dt-\partial_t\otimes ds,\label{eq:normalformA1}
\end{align}
where $\epsilon=1$ in case of positive signature and $\epsilon=-1$ in case of split signature.

\item Complex Liouville case: there are coordinates $z=x+iy,s,t$ and a holomorphic function $\rho(z)$ such that
\begin{align}
\begin{array}{c}
g=\frac{1}{4}(\bar{\rho}-\rho)(dz^2-d\bar{z}^2)+\frac{4}{\rho-\bar{\rho}}\left(\left(\frac{\partial\bar{\rho}}{\partial\bar{z}}\right)^2\left(ds+\rho\,dt\right)^2-\left(\frac{\partial \rho}{\partial z}\right)^2\left(ds+\bar{\rho}\,dt\right)^2\right)
\weg{g=2\mathcal{I}dxdy+\frac{4}{I}\left[2\mathcal{R}_x\mathcal{R}_y ds^2+2(\mathcal{I}(\mathcal{R}_x\mathcal{I}_y+\mathcal{R}_y\mathcal{I}_x)+2\mathcal{R}\mathcal{R}_x\mathcal{R}_y)dsdt+2(\mathcal{R}\mathcal{R}_x+\mathcal{I}\mathcal{I}_x)(\mathcal{R}\mathcal{R}_y+\mathcal{I}\mathcal{I}_y)dt^2\right]},\vspace{1mm}\\
\omega=\frac{\partial \rho}{\partial z}dz\wedge (ds+\bar{\rho}dt)+\frac{\partial \bar{\rho}}{\partial \bar{z}}d\bar{z}\wedge (ds+\rho dt)
\weg{J\partial_z=-\left(4\frac{\partial \rho}{\partial z}\right)^{-1}\left(\rho\,\partial_s-\partial_t\right),\,\,\,J\partial_{\bar{z}}=-\left(4\frac{\partial \rho}{\partial z}\right)^{-1}\left(\bar{\rho}\,\partial_s-\partial_t\right),\vspace{1mm}\\
J\partial_s=-\frac{4}{\rho-\bar{\rho}}\left(\frac{\partial\bar{\rho}}{\partial\bar{z}}\partial_{\bar{z}}-\frac{\partial \rho}{\partial z}\partial_{z}\right),\,\,\,J\partial_t=-\frac{4}{\rho-\bar{\rho}}\left(\frac{\partial\bar{\rho}}{\partial\bar{z}}\rho\partial_{\bar{z}}-\frac{\partial \rho}{\partial z}\bar{\rho}\partial_{z}\right),}
\end{array}\label{eq:normalform2}
\end{align}
and
\begin{align}
A=\rho\, \partial_z\otimes dz+\bar{\rho}\,\partial_{\bar{z}} \otimes d\bar{z}+(\rho+\bar{\rho})\partial_s\otimes ds+\rho\bar{\rho}\,\partial_s\otimes dt-\partial_t\otimes ds.\label{eq:normalformA2}
\end{align}

\item Degenerate case: there are coordinates $x,t,u_1,u_2$, a function $\rho(x)$ and a positively or negatively definite 2D-K\"ahler structure $(h,j,\Omega=h(j.,.))$ on the domain $\Sigma\subseteq \mathbb{R}^2$ of $u_1,u_2$ such that
\begin{align}
\begin{array}{c}
g=(c-\rho) h+(\rho-c)d x^2+\frac{(\rho')^2}{\rho-c}\theta^2,\,\,\,\omega=(c-\rho)\Omega+ \rho' dx \wedge \theta
\end{array}
\label{eq:normalform3}
\end{align}
and
\begin{align}
gA=c(c-\rho)h+\rho\frac{(\rho')^2}{\rho-c}\theta^2+\rho(\rho-c)d x^2,\label{eq:normalformA3}
\end{align}
where $\theta=dt-\tau$ and $\tau$ is a one-form on $\Sigma$ such that $d\tau=\Omega$.

\end{itemize}
Conversely, let $(g,\omega)$ and $A$ be given on some open subset of $\mathbb{R}^4$ as in one of the cases above. Then, $(g,J,\omega)$ defines a K\"ahler structure and $A\in \mathcal{A}(g,J)$.

\end{thm}
\begin{remark}
The K\"ahler structures $(g,\omega)$ in Theorem \ref{thm:normalformwithv} can be obtained as a special case of the metrics in Theorem \ref{thm:normalform} by a change of coordinates, i.e. in the Liouville case we define
$$
d x_{\mbox{\tiny new}}=\tfrac{1}{F(x)}d x, \quad d y_{\mbox{\tiny new}}=\tfrac{1}{G(y)}d y,
$$
in the complex Liouville case we define
$$
d z_{\mbox{\tiny new}}=\tfrac{1}{F(z)}d z,
$$
and in the degenerate case we define
$$
d x_{\mbox{\tiny new}}=\tfrac{1}{F(x)}d x,
$$
to pass from the coordinates in Theorem \ref{thm:normalform} to the coordinates in Theorem \ref{thm:normalformwithv}. 

Note that the formulas for $g$ and $A$ above yield the metrics \eqref{eq:ghat1}--\eqref{eq:ghat3} by solving equation \eqref{eq:defA} for $\hat{g}$.
\end{remark}

In what follows the number of non-constant eigenvalues of $A$ will be called the \emph{order} of $A$. We give the proof of Theorem \ref{thm:normalform} for any point from the open and dense subset $M^\circ \subset M$ introduced in \S\ref{sec:killing}.

\subsection{Proof of Theorem \ref{thm:normalform} for $A$ of order two}

Let $(M,g,J)$ be a K\"ahler surface and assume $A \in \mathcal{A}(g,J)$ has order two, i.e. it has two non-constant eigenvalues $\rho, \sigma$ so that, by Corollary \ref{cor:commutekilling}, the vector fields
\begin{align}
V_1=\grad_g(\rho+\sigma)\quad \mbox{ and }\quad V_2=\grad_g(\rho\sigma)\label{eq:killingexplicit}
\end{align}
span a rank $2$-subbundle $D\subset TM^{\circ}$ on the open dense subset $M^{\circ}\subset M$ and the orthogonal complement $D^{\perp}$ satisfies $D^{\perp}=JD$. Consequently, we have an orthogonal direct sum decomposition $TM^{\circ}=D\perp JD$ with respect to which the metric $g$ decomposes as
\begin{equation}\label{decompg}
g=g|_{D}\oplus g|_{JD}.
\end{equation}

Recall from Corollary \ref{cor:commutekilling} that the vector fields $V_i,K_i=JV_i$ all Lie commute and the subbundle $D$ is Frobenius integrable and gives rise to a codimension two foliation $\mathcal{F}$ on $M^\circ$ with totally geodesic leaves. On each leaf $\mathcal{L} \in \mathcal{F}$, the bundle metric $g|_{D}$ restricts to become a pseudo-Riemannian metric on $\mathcal{L}$ which agrees with the pullback $g_{\mathcal{L}}$ of $g$ to $\mathcal{L}$. Clearly, also $A$ decomposes as
\begin{equation}\label{decompA}
A=A|_{D}\oplus A|_{JD}.
\end{equation}
according to the decomposition of $TM$. Note that $A|_{D}$ is a bundle endomorphism of $D$ which restricts to each leaf $\mathcal{L} \in \mathcal{F}$ to become an endomorphism $A_{\mathcal{L}}$ of the tangent bundle of $\mathcal{L}$. Now it follows immediately from equation \eqref{eq:main} and the fact that $\mathcal{L}$ is totally geodesic that on each leaf $\mathcal{L} \in \mathcal{F}$ we have
\begin{equation}\label{eq:mainreal}
{}^{g_{\mathcal{L}}}\nabla_X{A_{\mathcal{L}}}=\frac{1}{2}\left(X^{\flat}\otimes \grad_{g_{\mathcal{L}}}\left(\tr A_{\mathcal{L}}\right)+\left(d \tr A_{\mathcal{L}}\right)^{\flat}\otimes X\right)
\end{equation}
for all $X \in \Gamma(T\mathcal{L})$ where $\flat$ is taken with respect to $g_{\mathcal{L}}$. Equation \eqref{eq:mainreal} is the real version of equation \eqref{eq:main}, i.e.~every solution $A_{\mathcal{L}}$ gives rise to a pseudo-Riemannian metric $\hat{g}_{\mathcal{L}}$ on $\mathcal{L}$ which has the same unparametrised geodesics as $g_{\mathcal{L}}$ (see for instance \cite[Theorem 2]{benenti}). In our case the eigenvalues of $A_{\mathcal{L}}$ are $\rho,\sigma$ (pulled back to $\mathcal{L}$), in particular $A_{\mathcal{L}}$ is non-parallel with respect to ${}^{g_{\mathcal{L}}}\nabla$. It follows that we can take advantage of the local classification of (real) projectively equivalent surface  {  metrics (see  the appendix of \cite{MR2892455} and the references therein for details).} 

Let us use the fact that the vector fields $V_1,V_2,K_1,K_2$ all Lie commute to introduce local coordinates $(\tilde{x},\tilde{y},s,t)$ in a neighborhood of every point $p \in M^{\circ}$ such that
$$
V_1=\partial_{\tilde{x}},\quad V_2=\partial_{\tilde{y}},\quad K_1=\partial_s,\quad K_2=\partial_t.
$$
Note that the matrix representation of the metric $g$ and the solution $A$ with respect to the coordinates $\tilde{x},\tilde{y},s,t$ decomposes into blocks
\begin{align}
g=\left(\begin{array}{cc}g^{\prime}&0\\0&g^{\prime}\end{array}\right),\quad A=\left(\begin{array}{cc}A^{\prime}&0\\0&A^{\prime}\end{array}\right)\label{eq:blockform}
\end{align}
where the matrix-valued functions $g^{\prime}, A^{\prime}$ do not depend on the coordinates $s,t$ and the eigenvalues of $A^{\prime}$ are $\rho,\sigma$.

We distinguish two cases  (see \eqref{eq:propA2}, the third case of a Jordan block has been already excluded):

\subsubsection{ Assume $\rho,\sigma$ are real}

{
Then, (see the appendix of \cite{MR2892455}),} there is a coordinate transformation $x=x(\tilde{x},\tilde{y}),y=y(\tilde{x},\tilde{y})$ such that with respect to the coordinates $x,y$, the pair $(g^{\prime},A^{\prime})$ becomes
\begin{align}
\begin{array}{cc}g_{\mbox{\tiny L}}=(\rho(x)-\sigma(y))\left(\begin{array}{cc}1&0\\0&\epsilon\end{array}\right),&A_{\mbox{\tiny L}}=\left(\begin{array}{cc}\rho(x)&0\\0&\sigma(y)\end{array}\right),
\end{array}\label{eq:Dini}
\end{align}
where $\epsilon=\pm 1$, depending on whether $g|_{D}$ has Riemannian or Lorentzian signature, and the non-constant eigenvalues $\rho,\sigma$ only depend on $x,y$ respectively.  {As in the appendix of \cite{MR2892455}, we will call   the coordinates $x,y$ ``Liouville coordinates''.}

From \eqref{eq:killingexplicit} and \eqref{eq:Dini}, we can compute $V_1,V_2$ in the Liouville coordinates $x,y$ and obtain
\begin{align}
V_1=\partial_{\tilde{x}}=\frac{1}{\rho-\sigma}\left(\rho'\partial_x+\epsilon\sigma'\partial_y\right),\quad V_2=\partial_{\tilde{y}}=\frac{1}{\rho-\sigma}\left(\rho' \sigma \partial_x+\epsilon\rho\sigma'\partial_y\right).\label{eq:v1v2}
\end{align}

Consequently, the differential $\frac{\partial (x,y)}{\partial(\tilde{x},\tilde{y})}$ of the coordinate transformation expressing $\tilde{x},\tilde{y}$ in terms of $x,y$ is given by
$$\frac{\partial (x,y)}{\partial(\tilde{x},\tilde{y})}(\tilde{x}(x,y),\tilde{y}(x,y))=\left(\begin{array}{cc}\frac{\partial x}{\partial\tilde{x}}&\frac{\partial x}{\partial\tilde{y}}\\\frac{\partial y}{\partial\tilde{x}}&\frac{\partial y}{\partial\tilde{y}}\end{array}\right)=\frac{1}{\rho-\sigma}\left(\begin{array}{cc}\rho'&\rho' \sigma\\\epsilon\sigma'&\epsilon\rho\sigma'\end{array}\right).$$

From this, we can calculate the matrix representations $g^{\prime},A^{\prime}$ of $g_{|D},A_{|D}$ in the coordinates $\tilde{x},\tilde{y}$ by applying the usual transformation rules. We obtain
$$g^{\prime}=\left(\frac{\partial (x,y)}{\partial(\tilde{x},\tilde{y})}\right)^T g_{\mbox{\tiny L}}\frac{\partial (x,y)}{\partial(\tilde{x},\tilde{y})}=\frac{1}{\rho-\sigma}\left(\begin{array}{cc}(\rho')^2+\epsilon(\sigma')^2&(\rho')^2\sigma+\epsilon\rho (\sigma')^2\\(\rho')^2\sigma+\epsilon\rho (\sigma')^2&(\rho')^2\sigma^2+\epsilon\rho^2 (\sigma')^2\end{array}\right),$$
$$A^{\prime}=\left(\frac{\partial (x,y)}{\partial(\tilde{x},\tilde{y})}\right)^{-1}A_{\mbox{\tiny L}}\frac{\partial (x,y)}{\partial(\tilde{x},\tilde{y})}=\left(\begin{array}{cc}\rho+\sigma&\rho\sigma\\-1&0\end{array}\right).$$

Using these equations together with \eqref{eq:Dini} and \eqref{eq:v1v2}, we see that in the coordinates $x,y,s,t$ the K\"ahler structure $(g,J)$ takes the form \eqref{eq:normalform1} and $A$ is given by \eqref{eq:normalformA1}.

\subsubsection{$\sigma$ and $\rho$ are complex conjugates}

Writing $\mathcal{R}=\frac{1}{2}\left(\rho+\sigma\right)$ and $\mathcal{I}=\frac{1}{2i}\left(\rho-\sigma\right)$ we have
\begin{align}
K_1=2J\grad(\mathcal{R})\quad \mbox{and} \quad K_2=J\grad(\mathcal{R}^2+\mathcal{I}^2).\label{eq:killingexplicitcc}
\end{align}
In this case, we can introduce so-called complex Liouville coordinates $x=x(\tilde{x},\tilde{y}),y=y(\tilde{x},\tilde{y})$ (see ~\cite{MR2892455}) in which the corresponding matrices of $g^{\prime}$ and $A^{\prime}$ take the form
\begin{align}
\begin{array}{cc}
g_{\mbox{\tiny CL}}=\left(\begin{array}{cc}0&\mathcal{I}\\\mathcal{I}&0\end{array}\right),&A_{\mbox{\tiny CL}}=\left(\begin{array}{cc}\mathcal{R}&-\mathcal{I}\\\mathcal{I}&\mathcal{R}\end{array}\right),
\end{array}\label{eq:cplxLiouville}
\end{align}
where $\rho(z)=\mathcal{R}(z)+i\mathcal{I}(z)$ is a holomorphic function of the complex variable $z=x+iy$.

Writing $\mathcal{R}_x=\frac{\partial\mathcal{R}}{\partial x}$ etc., from \eqref{eq:killingexplicitcc} and \eqref{eq:cplxLiouville} we obtain
\begin{align}
\frac{\partial}{\partial \tilde{x}}=\frac{2}{\mathcal{I}}\left(\mathcal{R}_y \frac{\partial}{\partial x}+\mathcal{R}_x\frac{\partial}{\partial y}\right),\,\,\, \frac{\partial}{\partial \tilde{y}}=\frac{2}{\mathcal{I}}\left((\mathcal{R}\mathcal{R}_y+\mathcal{I}\mathcal{I}_y)\frac{\partial}{\partial x}+(\mathcal{R}\mathcal{R}_x+\mathcal{I}\mathcal{I}_x)\frac{\partial}{\partial y}\right).\label{eq:v1v2cc}
\end{align}

Consequently, the differential $\frac{\partial (x,y)}{\partial(\tilde{x},\tilde{y})}$ of the coordinate transformation expressing $\tilde{x},\tilde{y}$ in terms of $x,y$ is given by
$$\frac{\partial (x,y)}{\partial(\tilde{x},\tilde{y})}(\tilde{x}(x,y),\tilde{y}(x,y))=\left(\begin{array}{cc}\frac{\partial x}{\partial\tilde{x}}&\frac{\partial x}{\partial\tilde{y}}\\\frac{\partial y}{\partial\tilde{x}}&\frac{\partial y}{\partial\tilde{y}}\end{array}\right)=\frac{2}{\mathcal{I}}\left(\begin{array}{cc}\mathcal{R}_y&\mathcal{R}\mathcal{R}_y+\mathcal{I}\mathcal{I}_y\\\mathcal{R}_x &\mathcal{R}\mathcal{R}_x+\mathcal{I}\mathcal{I}_x\end{array}\right).$$

From this, we can calculate the matrix representations $g^{\prime},A^{\prime}$ of $g_{|D},A_{|D}$ in the coordinates $\tilde{x},\tilde{y}$ by applying the usual transformation rules. Using the Cauchy-Riemann equations $\mathcal{R}_x=\mathcal{I}_y,\mathcal{R}_y=-\mathcal{I}_x$ we obtain
$$
\aligned
g^{\prime}&=\frac{\partial (x,y)}{\partial(\tilde{x},\tilde{y})}^T g_{\mbox{\tiny CL}}\frac{\partial (x,y)}{\partial(\tilde{x},\tilde{y})}\\&=\frac{4}{\mathcal{I}}\left(\begin{array}{cc}2\mathcal{R}_x\mathcal{R}_y&\mathcal{I}(\mathcal{R}_x\mathcal{I}_y+\mathcal{R}_y\mathcal{I}_x)+2\mathcal{R}\mathcal{R}_x\mathcal{R}_y\\\mathcal{I}(\mathcal{R}_x\mathcal{I}_y+\mathcal{R}_y\mathcal{I}_x)+2\mathcal{R}\mathcal{R}_x\mathcal{R}_y&2(\mathcal{R}\mathcal{R}_x+\mathcal{I}\mathcal{I}_x)(\mathcal{R}\mathcal{R}_y+\mathcal{I}\mathcal{I}_y)\end{array}\right),
\endaligned
$$
$$A^{\prime}=\frac{\partial (x,y)}{\partial(\tilde{x},\tilde{y})}^{-1}A_{\mbox{\tiny CL}}\frac{\partial (x,y)}{\partial(\tilde{x},\tilde{y})}=\left(\begin{array}{cc}2\mathcal{R}&\mathcal{R}^2+\mathcal{I}^2\\-1&0\end{array}\right).$$

Using these equations together with \eqref{eq:cplxLiouville} and \eqref{eq:v1v2cc}, the K\"ahler structure $(g,J)$ in the coordinates $z=x+iy,s,t$ takes the form \eqref{eq:normalform2} and $A$ is given by \eqref{eq:normalformA2}.

\subsection{The case when $A$ has order one}

Let us now present the normal forms of a $4$-dimensional  K\"ahler structure $(g,J)$ admitting $A\in\mathcal{A}(g,J)$, having one non-constant eigenvalue $\rho$ and a constant eigenvalue $c$. Here the procedure of extending the projective setting to the c-projective one, as it was applied in the last two subsections, does not work. Indeed, in the case that $A$ has order one, the distribution $D$ is spanned by a single vector field $V_1=\grad_g\,\rho$ and the corresponding one-dimensional block does not determine the whole of the K\"ahler structure as it was the case for Liouville coordinates \eqref{eq:normalform1}, \eqref{eq:normalformA1} and complex Liouville coordinates \eqref{eq:normalform2}, \eqref{eq:normalformA2}. We will therefore use the results of Apostolov et al \cite{MR2228318} which we shall describe briefly in what follows. Note that these results have been obtained for Riemannian signature only. However, in this special case the proof does not need any modifications to work in the pseudo-Riemannian  setting as well.

As before, let $M^{\circ}$ be the open and dense subset of the points of $M$ where  $d\rho\neq 0$ and $\rho \ne c$. Then, according to ~\cite{MR2228318}, we can introduce coordinates $\rho,t,u_1,u_2$, in a neighborhood of every point of $M^{\circ}$ such that $K_1=\partial_t$ and the K\"ahler structure $(g,J)$ takes the form
\begin{align}
\begin{array}{c}
g=(c-\rho) h+\frac{\rho-c}{f(\rho)}d\rho^2+\frac{f(\rho)}{\rho-c}\theta^2,\quad d \rho\circ J=-\frac{f(\rho)}{\rho-c}\theta, \quad \theta\circ J=\frac{\rho-c}{f(\rho)}d \rho,
\end{array}
\label{eq:normalform3a}
\end{align}
where $f$ is a function of one variable, $(h,j)$ is a positively or negatively definite K\"ahler structure on the domain $\Sigma\subseteq \mathbb{R}^2$ of the coordinate functions $u_1,u_2$ and $\theta=dt-\tau$ for a $1$-form $\tau$ on $\Sigma$ satisfying $d\tau=h(j.,.)$.

Moreover, $A$ in these coordinates takes the form
\begin{align}
gA=c(c-\rho)h+\rho\frac{f(\rho)}{\rho-c}\theta^2+\rho\frac{\rho-c}{f(\rho)}d\rho^2.\label{eq:normalformA3a}
\end{align}

To keep consistent with the notation of the previous sections, we introduce a new coordinate $x=x(\rho)$ by requiring $dx=\frac{1}{\sqrt{f}}d\rho$. In the coordinates $x,t,u_1,u_2$, the K\"ahler structure $(g,J)$ takes the form \eqref{eq:normalform3} and $A$ is given by \eqref{eq:normalformA3}.


\section{Proof of Theorem \ref{thm:normalformwithv}}\label{sec:solvinglie}


Recall that given an essential c-projective vector field $v$ for $(g,J)$, we always have a non-trivial solution $A\in \mathcal{A}(g,J)$, hence, the degree of mobility $d(g,J)$ is at least two. On the other hand, we have

\begin{prop}\label{prop:degree2indim2}
Let $(M,g,J)$ be a connected K\"ahler surface (of arbitrary signature)  of non-constant holomorphic sectional curvature. Then, $d(g,J)\leq 2$.
\end{prop}

This statement was proven in ~\cite{MR2144249} for Riemannian and in~\cite{MR2948791} for arbitrary signature.

\begin{remark}
Proposition \ref{prop:degree2indim2} holds true in higher dimensions as well if $M$ is assumed to be closed~\cite{MR2948791}, but fails to be true without the closedness assumption \cite{CONE}.
\end{remark}

Thus, since we are working in the situation of K\"ahler surfaces, we can restrict to the case $d(g,J)=2$ in what follows.


\subsection{K\"ahler metrics of degree of mobility two admitting c-projective vector fields}


\subsubsection{{ Equations on $g, A$ and $v$.}}  

Suppose that $(M,g,J)$ is a K\"ahler surface with $d(g,J)=2$ and  $h,\hat{h}\in \mathcal{S}([\gnabla],J)$ is a basis of the space of solutions to \eqref{eq:maininvariant}. Let $v$ be a c-projective vector field. Since the Lie derivative $\mathcal{L}_v$ gives an endomorphism
$$
\mathcal{L}_v:\mathcal{S}([\gnabla],J)\rightarrow \mathcal{S}([\gnabla],J)
$$
we can write
\begin{align}
\begin{array}{c}
\mathcal{L}_v h =\gamma h +\delta \hat{h},\quad \mathcal{L}_v \hat{h}=\alpha h+\beta \hat{h},
\end{array}\label{eq:matrixv}
\end{align}
for certain constants $\alpha,\beta,\gamma,\delta$.

In the case when one of the basis vectors, say $h$, is non-degenerate on some open subset $U\subseteq M$ we can think of it as arising from a metric $g$ on $U$.  Then, as explained in \S\ref{sec:compks},   instead of $h$ and $\hat{h}$  we  can equivalently work with the metric $g$ and an endomorphism $A\in \mathcal{A}(g,J)$ being a solution of  \eqref{eq:main}. In particular, \eqref{eq:matrixv} is equivalent to a \textsc{pde} system of first order on $g$, $A$ and  $v$ as the next lemma shows.

\begin{lem}\label{key}
Let $(M,g,J)$ be a K\"ahler manifold such that the degree of mobility $d(g,J)$ is two and let $v$ be a c-projective vector field. Then, for every $A\in\mathcal{A}(g,J)$, $A\neq \mathrm{const}\cdot \Id$, there exist constants $\alpha,\beta,\gamma,\delta$ such that
\begin{align}
\mathcal{L}_v g=-\delta gA-\left(\tfrac{1}{2}\delta\tr A+(n+1)\gamma\right)g\label{eq:PDEvgJ}
\end{align}
and
\begin{align}
\mathcal{L}_v A=-\delta A^2+(\beta-\gamma) A+\alpha\, \mathrm{Id}.\label{eq:LieA}
\end{align}
\end{lem}

\begin{proof} Let $A\in\mathcal{A}(g,J)$, $A\neq \mbox{const}\cdot \Id$. As we explained above, the spaces $\mathcal{A}(g,J)$ and $\mathcal{S}([\gnabla],J)$ are isomorphic via the map $\varphi_g$ defined by \eqref{eq:solspaceiso}. Thus, we can define basis vectors $h, \hat{h} \in \mathcal{S}([\gnabla],J)$ from the relations
 $\Id = \varphi_g (h)$ and $A=\varphi_g(\hat{h})$ so that
 \begin{equation}\label{eq:defA0}
h = h_g = g^{-1}\otimes (\det g)^{1/(2n+2)}  \quad \mbox{and} \quad
A = \hat h h^{-1}.
\end{equation}

The Lie derivatives of $h$ and $\hat {h}$ along $v$ can be written as in \eqref{eq:matrixv} for certain constants $\alpha,\beta,\gamma,\delta$. As $A$ and $g$ are related to $h$ and $\hat h$ by means of explicit formulas \eqref{eq:defA0}, we can easily find their Lie derivatives too.

Taking the Lie derivative of $A=  \hat h h^{-1}$ immediately yields equation \eqref{eq:LieA}.

Furthermore, writing out the Lie derivative of $h=h_g$ explicitly in terms of $g$ by using \eqref{eq:defA0}, we obtain that the first equation in \eqref{eq:matrixv} is equivalent to
\begin{equation}\label{someid}
 g^{-1}\mathcal{L}_v g-\frac{\tr( g^{-1}\mathcal{L}_v g)}{2(n+1)}\Id=-\delta A-\gamma \Id.
\end{equation}

Taking the trace gives
$$
\frac{1}{n+1}\tr(g^{-1}\mathcal{L}_v g)=-\delta \tr(A)-2n\gamma
$$ and inserting this back into \eqref{someid} yields \eqref{eq:PDEvgJ}.
\end{proof}

We can now insert the normal forms for $(g,J)$ and $A$, as they have been obtained in Theorem \ref{thm:normalform}, into \eqref{eq:PDEvgJ},\eqref{eq:LieA} and obtain a \textsc{pde} system on the unspecified data appearing in the normal forms for $g,A$ (for example, the functions $\rho,\sigma$ in the Liouville case) and on the components of the c-projective vector field. This reduces our problem to solving a system of \textsc{ode}s.

The integration of these \textsc{ode}s, which depend on the constants $\alpha,\beta,\gamma,\delta$, can be simplified further by choosing a special basis $h,\hat{h}$ of $\mathcal{S}([\gnabla],J)$ in which the constants take a special form.

\subsubsection{Normal forms for $\mathcal{L}_v$}

Suppose that $v$ is essential for at least one metric in the c-projective class of $g$. Then,  there are, up to rescaling of $v$, four possible normal forms for the matrix representation
\begin{align}
\left(\begin{array}{cc}
\gamma&\alpha\\\delta&\beta
\end{array}\right)\label{eq:matrixformv}
\end{align}
of the endomorphism $\mathcal{L}_v:\mathcal{S}([\gnabla],J)\rightarrow \mathcal{S}([\gnabla],J)$ in a basis $h,\hat{h}$ of $\mathcal{S}([\gnabla],J)$ as in \eqref{eq:matrixv}:
\begin{align}
\left(\begin{array}{cc}
0&1\\0&0
\end{array}\right),\left(\begin{array}{cc}
1&0\\0&\beta
\end{array}\right)\mbox{ for certain }\beta\neq 1,\left(\begin{array}{cc}
1&1\\0&1
\end{array}\right),\left(\begin{array}{cc}
\beta&-1\\1&\beta
\end{array}\right)\mbox{ for certain }\beta.
\label{eq:normalformsv}
\end{align}

To use these normal forms in Lemma \ref{key} and in the \textsc{pde}s \eqref{eq:PDEvgJ},\eqref{eq:LieA}, we must show that for each choice of basis $h,\hat{h}$, at least one of these vectors is non-degenerate and can therefore be viewed as arising from a metric.

Recall that the order of $A\in \mathcal{A}(g,J)$ is the number of non-constant eigenvalues of $A$. In the case $d(g,J)=2$, this is obviously an invariant of the c-projective class $[g]$: if $A$ has order $l$, then every $A'\in \mathcal{A}(g',J)$, $A'\neq \mbox{const}\cdot \Id$ has order $l$ for every $g'\in [g]$ (see also Lemma \ref{lem:aboutL} below).

\begin{lem}\label{lem:nondeg}
Let $(M,g,J)$ be a connected K\"ahler manifold. If $h\in\mathcal{S}([\gnabla],J)$ is non-degenerate at a point $p\in M$, it is non-degenerate on a dense open subset of $M$.
\end{lem}

\begin{proof}
Let $A\in \mathcal{A}(g,J)$ such that $\varphi_g(h)=A$. Then, $A$ is non-degenerate at $p$, in particular, the constant eigenvalues of $A$ are non-zero. Since the differentials of the nonconstant eigenvalues of $A$ are non-zero on the open and dense subset $M^{\circ}$  (see ~\cite{MR2228318} or Corollary \ref{cor:commutekilling}), the nonconstant eigenvalues cannot be equal to zero on an open subset.
\end{proof}

\begin{lem}\label{lem:order2}
Let $(M,g,J)$ be a connected K\"ahler surface such that $d(g,J)=2$ and assume that the c-projective class of $g$ has order two.
Then every non-zero $h\in \mathcal{S}([\gnabla],J)$ is non-degenerate on an open and dense subset of $M$ and hence, comes from a K\"ahler metric $\hat{g}$ on this subset (such that $h=h_{\hat{g}}$) which is c-projectively equivalent to $g$.
\end{lem}

\begin{proof}
A non-zero $h\in\mathcal{S}([\gnabla],J)$ corresponds to a non-zero $A\in \mathcal{A}(g,J)$ via $A=\varphi_g(h)$. This $A$ is either a non-zero multiple of the identity (in which case it is nondegenerate at every point) or has two non-constant eigenvalues. These eigenvalues cannot vanish on an open subset since their differentials do not vanish on an open and dense subset of $M^\circ\subset M$.
\end{proof}

The statement above is not true when the order is one. However, we have
\begin{lem}\label{lem:order1}
Let $(M,g,J)$ be a connected K\"ahler surface such that $d(g,J)=2$ and assume that the c-projective class of $g$ has order one.

Given a basis $h,\hat{h}\in \mathcal{S}([\gnabla],J)$, at least one of these vectors is non-degenerate on an open dense subset of $M$ and hence, comes from a K\"ahler metric $\hat{g}$ on this subset which is c-projectively equivalent to $g$.

Moreover, suppose in the basis $h,\hat{h}$, the endomorphism $\mathcal{L}_v$ takes one of the normal forms in \eqref{eq:normalformsv}. Then we have the following cases:
\begin{itemize}
\item Let $\left(\begin{array}{cc}
\gamma&\alpha\\\delta&\beta
\end{array}\right)=\left(\begin{array}{cc}
0&1\\0&0\end{array}\right).$ Then, $h$ is degenerate.
\item Let $\left(\begin{array}{cc}
\gamma&\alpha\\\delta&\beta
\end{array}\right)=\left(\begin{array}{cc}
1&0\\0&\beta\end{array}\right)\mbox{ for }\beta\neq 0.$ Then, exactly one of the $h,\hat{h}$ is non-degenerate, the other is degenerate.
\item Let $\left(\begin{array}{cc}
\gamma&\alpha\\\delta&\beta
\end{array}\right)=\left(\begin{array}{cc}
1&1\\0&1\end{array}\right).$ Then, $h$ is degenerate.
\item The case $\left(\begin{array}{cc}
\gamma&\alpha\\\delta&\beta
\end{array}\right)=\left(\begin{array}{cc}
\beta&-1\\1&\beta\end{array}\right)$ cannot occur.
\end{itemize}
\end{lem}

\begin{proof}
Let $A,\hat{A}\in \mathcal{A}(g,J)$ correspond to $h,\hat{h}$, i.e. $\varphi_g(h)=A,\varphi_g(\hat{h})=\hat{A}$. Since $h,\hat{h}$ form a basis, we have $\hat{A}=c_1 A+c_2\Id$ for certain constants $c_1,c_2$ where $c_2\neq 0$. Clearly, if one of the endomorphisms $A$ or $\hat{A}$ is degenerate, i.e. has a constant eigenvalue equal to zero, the other endomorphism is non-degenerate on a dense and open subset.

Let us now consider the normal forms in \eqref{eq:normalformsv}. If $h$ is non-degenerate, i.e. $h=h_{g'}$ for certain $g'\in [g]$, equation \eqref{eq:LieA} implies that the constant eigenvalue $c$ of $A'=\varphi_{g'}(\hat{h})$ must satisfy the equation
\begin{align}
0=-\delta c^2+(\beta-\gamma)c+\alpha.\label{eq:equationonconstev}
\end{align}
Suppose that in the first case in \eqref{eq:normalformsv}, $h$ is non-degenerate. Then, by \eqref{eq:equationonconstev}, the constant eigenvalue of $A'$ must satisfy $0=1$ which is a contradiction.  

Suppose that in the second case in \eqref{eq:normalformsv}, $h$ is non-degenerate. Then, by \eqref{eq:equationonconstev}, the constant eigenvalue of $A'$ must satisfy $0=(\beta-1)c$ for $\beta\neq 1$, hence, $c=0$. Thus, $A'$ is degenerate and therefore also $\hat{h}$ is degenerate.

Suppose that in the third case in \eqref{eq:normalformsv}, $h$ is non-degenerate. Using \eqref{eq:equationonconstev}, we again obtain a contradiction.

Suppose that in the fourth case in \eqref{eq:normalformsv}, $h$ is non-degenerate. Then, \eqref{eq:equationonconstev} does not have any real solutions. This contradicts to the fact that $A'$ has a constant real eigenvalue.  The same argument works when $\hat{h}$ is assumed to be non-degenerate. 
\end{proof}

Thus, we can simplify the integration of the \textsc{ode}s obtained from  \eqref{eq:PDEvgJ},\eqref{eq:LieA} after inserting the normal forms from Theorem \ref{thm:normalform}, by first choosing an appropriate basis $h,\hat{h}$ of $\mathcal{S}([\gnabla],J)$ in which the matrix of $\mathcal{L}_v$ takes one of the normal forms in \eqref{eq:normalformsv}. Then, by Lemma \ref{lem:order2} and Lemma \ref{lem:order1}, either $h$ or $\hat{h}$ (or both) can be viewed as arising from a metric, therefore, we can solve \eqref{eq:PDEvgJ},\eqref{eq:LieA} with respect to this new metric for the specified constants $\alpha,\beta,\gamma,\delta$.


\subsection{The case when $A\in\mathcal{A}(g,J)$ has order two}\label{sec:nonconst}


\subsubsection{Decomposition of the components of the c-projective vector field}

Let $(M, g, J)$ be a K\"ahler surface and let $A\in \mathcal A(g,J)$ be a solution of \eqref{eq:main} {of order two. Let $\rho$ and $\sigma$ be the two (real or conjugate complex) nonconstant eigenvalues of $A$. We are working on the open and dense subset $M^\circ \subset M' \subset M$ where the distribution $D$ spanned by the vector fields $V_1=-JK_1$ and $V_2=-JK_2$ (where $K_1=J\operatorname{grad}_g(\rho+\sigma),K_2=J\operatorname{grad}_g(\sigma\rho)$ are the Killing vector fields from Corollary \ref{cor:commutekilling}) has rank equal to two.} Recall that we have the orthogonal direct sum decomposition (on $M^\circ \subset M$) 
\begin{align}
TM=D\perp  JD.\label{eq:decomptangent}
\end{align}
and the annihilator of $JD$ in $T^* M$ is given by $D^*=\mathrm{span}\{d(\rho+\sigma),d(\rho\sigma)\}$.

Since we want to have the possibility to replace $g$ with any other metric $g'\in [g]$ and, correspondingly, $A$ with  $A'\in \mathcal A(g',J)$, we need to make sure that all the above geometric objects (namely, $M^\circ$, $M'$, $D$, $JD$ and $D^*$) remain unchanged under such an operation, i.e., are projectively invariant.  Since all of them are defined in terms of the eigenvalues of $A$  and the complex structure $J$, it is sufficient to describe the relationship between $A$ and $A'$.

\begin{lem}
\label{lem:aboutL}
Let $g'\in [g]$, $A'\in \mathcal A(g',J)$ and $d(g,J)=2$. Then $A'=(c A + d\, \mathrm{Id})(aA+b \, \mathrm{Id})^{-1}$ for some constants $a,b,c,d\in \R$. In particular,
$$
\rho'=\frac{c \rho + d}{a\rho + b} \quad\mbox{and} \quad \sigma'=\frac{c\sigma  + d}{a\sigma  + b}
$$
where $\rho', \sigma'$ are the eigenvalues of $A'$.
\end{lem}

\begin{proof}
Let $h,\hat{h}$ be the basis of $\mathcal{S}([\gnabla],J)$ such that
$$
h=h_g=g^{-1}\otimes (\mathrm{det}\,g)^{\tfrac{1}{2(n+1)}}\quad \mbox{and} \quad A=\varphi_g(\hat{h})=\hat{h}h^{-1}.
$$

Choosing $g'\in [g]$, $A'\in \mathcal A(g',J)$ leads to another basis $q, \hat q\in \mathcal{S}([\gnabla],J)$ defined in a similar way: 
$$
q=h_{g'} \quad \mbox{and} \quad  A'=\varphi_{g'}(\hat{q})=\hat{q}q^{-1}.
$$

 Then, for certain real numbers $a,b,c,d$ we have $q=b h+a \hat{h}$ and $\hat{q}=d h+c\hat{h}$ and hence
$$
A'=(d h+c\hat{h})(b h+a \hat{h})^{-1} = (c A + d\, \mathrm{Id})(aA+b \, \mathrm{Id})^{-1},
$$ 
as required.
\end{proof}

Thus, the eigenvalues of $A'$ as functions on $M$ behave in a similar way as those of $A$. In particular, $\rho'\ne\sigma'$ if and only if $\rho\ne\sigma$. Also the differentials $d\rho$ and $d\rho'$  are proportional (with nonzero coefficient) so that the distribution  $D^*= \mathrm{span}\{d(\rho+\sigma),d(\rho\sigma)\}$  (and therefore $JD$ and $D$) remain unchanged if we replace $\rho$ and $\sigma$ with $\rho'$ and $\sigma'$. Thus, we get 

\begin{cor}\label{cor:bols1}
The subsets $M'$ and $M^\circ$ and the distributions $D^*$, $D$ and $JD$ are c-projectively invariant. 
\end{cor}

\begin{remark}
Lemma \ref{lem:aboutL} does not use the fact that $A$ has order two,  so  $M'$, $M^\circ$ and $D^*$ are c-projectively invariant in both cases,  i.e., for $A$ of order $1$ and $2$.  
\end{remark}

This corollary immediately implies that the c-projective vector field $v$ splits into two independent components with respect to the decomposition \eqref{eq:decomptangent} of $TM$ in the following sense.  Let $x,y,s,t$ be the coordinates in which $(g,J)$ and $A$ take the Liouville form \eqref{eq:normalform1},\eqref{eq:normalformA1} or the complex Liouville form \eqref{eq:normalform2},\eqref{eq:normalformA2}. Since these coordinates are adapted to the decomposition \eqref{eq:decomptangent}, we  obtain

\begin{cor}
\label{cor:hprojectiveinvariance}
In the coordinates $x,y,s,t$, every c-projective vector field splits into two independent components
$$
v=\underbrace{v^x(x,y)\partial_x+v^y(x,y)\partial_y}_{=:v_D}+\underbrace{v^s(s,t)\partial_s+v^t(s,t)\partial_t}_{=:v_{JD}}.
$$
\end{cor}

Let us write $g_D,g_{JD}$ for the corresponding blocks of the metric. Using Corollary \ref{cor:hprojectiveinvariance} and the fact that $\partial_{s},\partial_{t}$ are Killing vector fields, it follows that $\mathcal{L}_v g$ decomposes into blocks
\begin{align}
\mathcal{L}_v g=\underbrace{\mathcal{L}_{v_D} g_{D}}_{\mbox{\tiny upper-left}}+\underbrace{\mathcal{L}_{v_D} g_{JD}}_{\mbox{\tiny lower-right}}+\underbrace{\mathcal{L}_{v_{JD}} g_{JD}}_{\mbox{\tiny lower-right}},\label{eq:blockdecomp}
\end{align}
where ``upper-left'' respectively ``lower-right'' refers to the blocks spanned by the pairs of forms $d x,d y$ and $d s,d t$ respectively.

Consequently, the \textsc{pde}s \eqref{eq:PDEvgJ},\eqref{eq:LieA} split into an upper-left block and a lower-right block. The upper-left block is a system of \textsc{pde}s of first order in the two independent variables $x,y$ on the functions $v^x(x,y),v^y(x,y)$ and the parameters of the metric $\rho(x),\sigma(y)$ in the Liouville case \eqref{eq:normalform1} and $\rho(z)$ in the complex Liouville case \eqref{eq:normalform2}. Solving this system, we can insert the obtained quantities in the lower-right block of \eqref{eq:PDEvgJ},\eqref{eq:LieA} which gives a linear \textsc{pde} system in two independent variables $s,t$ on the functions $v^s(s,t),v^t(s,t)$ with coefficients that do not depend on $s,t$.

\subsubsection{The Liouville case: Solving the \textsc{pde} system \eqref{eq:PDEvgJ} and \eqref{eq:LieA}}

Let us work in the Liouville coordinates \eqref{eq:normalform1},\eqref{eq:normalformA1} and recall that the \textsc{pde}s \eqref{eq:PDEvgJ} and \eqref{eq:LieA} split into two blocks of equations.

By direct calculation, we obtain that the upper-left block of \eqref{eq:PDEvgJ} is equivalent to the \textsc{pde}s
\begin{align}
\begin{array}{c}
v^x\frac{\partial\rho}{\partial x}-v^y\frac{\partial\sigma}{\partial y}+2(\rho-\sigma)\frac{\partial v^x}{\partial x}=-(\rho-\sigma)\left(\delta(2\rho+\sigma)+3\gamma\right),\vspace{1mm}\\
v^x\frac{\partial\rho}{\partial x}-v^y\frac{\partial\sigma}{\partial y}+2(\rho-\sigma)\frac{\partial v^y}{\partial y}=-(\rho-\sigma)\left(\delta(\rho+2\sigma)+3\gamma\right),\vspace{1mm}\\
\frac{\partial v^x}{\partial y}+\epsilon\frac{\partial v^y}{\partial x}=0.
\end{array}\label{eq:ULBLiouville1}
\end{align}

Similarly, we obtain that the upper-left block of \eqref{eq:LieA} is equivalent to the \textsc{pde}s
\begin{align}
\begin{array}{c}
v^x\frac{\partial\rho}{\partial x}=-\delta \rho^2+(\beta-\gamma)\rho+\alpha,\quad v^y\frac{\partial\sigma}{\partial y}=-\delta \sigma^2+(\beta-\gamma)\sigma+\alpha,\vspace{1mm}\\
\frac{\partial v^x}{\partial y}=\frac{\partial v^y}{\partial x}=0.
\end{array}\label{eq:ULBLiouville2}
\end{align}

We see that, since $v^x,\rho$ only depend on $x$ and $v^y,\sigma$ only depend on $y$, the equations \eqref{eq:ULBLiouville1},\eqref{eq:ULBLiouville2} give us a system of \textsc{ode}s.

Moreover, we can simplify \eqref{eq:ULBLiouville1} by inserting \eqref{eq:ULBLiouville2} and are left with the \textsc{ode} system
\begin{align}
\begin{array}{c}
v^x\frac{\partial\rho}{\partial x}=-\delta \rho^2+(\beta-\gamma)\rho+\alpha,\quad v^y\frac{\partial\sigma}{\partial y}=-\delta \sigma^2+(\beta-\gamma)\sigma+\alpha,\vspace{1mm}\\
2\frac{\partial v^x}{\partial x}=-\delta\rho-\beta-2\gamma,\quad 2\frac{\partial v^y}{\partial y}=-\delta\sigma-\beta-2\gamma.
\end{array}\label{eq:ULBLiouville3}
\end{align}

Since the functions $v^x,v^y$ are non-vanishing on a dense open subset of the coordinate neighborhood we are working in, we can introduce new coordinates $x_1,x_2$ be requiring
$d x_1=\frac{1}{v^x}d x,\,\,\,d x_2=\frac{1}{v^y}d y.$ In these coordinates we have
$$
v_D=\partial_{x_1}+\partial_{x_2},
$$
and the \textsc{ode} system  \eqref{eq:ULBLiouville3} can be written as

\begin{align}
\begin{array}{c}
\frac{\partial\rho}{\partial x_1}=-\delta \rho^2+(\beta-\gamma)\rho+\alpha,\quad \frac{\partial\sigma}{\partial x_2}=-\delta \sigma^2+(\beta-\gamma)\sigma+\alpha,\vspace{1mm}\\
\frac{2}{v^x}\frac{\partial v^x}{\partial x_1}=-\delta\rho-\beta-2\gamma, \quad
\frac{2}{v^y}\frac{\partial v^y}{\partial x_2}=-\delta\sigma-\beta-2\gamma.
\end{array}\label{eq:ODEsLiouville}
\end{align}

First let us solve the lower-right block of \eqref{eq:PDEvgJ} on the unknown components $v^s(s,t),v^t(s,t)$ of the c-projective vector field $v$. This can be done by replacing the derivatives of $\rho,\sigma,v^x,v^y$ by the equations \eqref{eq:ODEsLiouville} - we do not need to know the explicit solutions to \eqref{eq:ODEsLiouville}. Recall that the lower-right block of \eqref{eq:PDEvgJ}  reads
$$
\mathcal{L}_{v_{JD}}g_{JD}=-\mathcal{L}_{v_{D}}g_{JD}-\delta g_{JD}A_{JD}-(\delta(\rho+\sigma)+3\gamma)g_{JD},
$$
where $v_D=\partial_{x_1}+\partial_{x_2}$ and $v_{JD}=v^s\partial_s+v^t\partial_t$.

Writing for short $g_{st}=g(\partial_s,\partial_t)$ etc., a straightforward calculation yields that the left-hand side of the above equation is given by
\begin{align}
\begin{array}{c}
\mathcal{L}_{v_{JD}}g_{JD}=2\left(\frac{\partial v^s}{\partial s} g_{ss}+\frac{\partial v^t}{\partial s}g_{st}\right)ds^2+2\left(\frac{\partial v^s}{\partial t}g_{st}+\frac{\partial v^t}{\partial t}g_{tt}\right)dt^2\vspace{1mm}\\
+2\left(\frac{\partial v^s}{\partial t}g_{ss}+\frac{\partial v^t}{\partial s}g_{tt}+\left(\frac{\partial v^s}{\partial s}+\frac{\partial v^t}{\partial t}\right)g_{st}\right)dsdt.
\end{array}\label{eq:LHSPDE}
\end{align}

On the other hand, using the equations in \eqref{eq:ODEsLiouville}, a straightforward calculation shows that the right-hand side of the above \textsc{pde} is equal to
\begin{align}
\begin{array}{c}
-\mathcal{L}_{v_{D}}g_{JD}-\delta g_{JD}A_{JD}-(\delta(\rho+\sigma)+3\gamma)g_{JD}\vspace{1mm}\\
=2\left(-(\beta+2\gamma)g_{ss}-\delta g_{st}\right)ds^2+2\left(-\alpha g_{st}-(2\beta+\gamma)g_{tt}\right)dt^2+\vspace{1mm}\\
+2\left(-\alpha g_{ss}-\delta g_{tt}-3(\beta+\gamma)g_{st}\right)dsdt.
\end{array}\label{eq:RHSPDE}
\end{align}

Comparing \eqref{eq:LHSPDE} and \eqref{eq:RHSPDE} and using the non-degeneracy of $g_{JD}$, we obtain the equations
$$
\frac{\partial v^s}{\partial s}=-(\beta+2\gamma),\,\,\,\frac{\partial v^t}{\partial s}=-\delta,\,\,\,\frac{\partial v^s}{\partial t}=-\alpha,\,\,\,\frac{\partial v^t}{\partial t}=-(2\beta+\gamma).
$$

It follows that, up to adding constant linear combinations of the Killing vector fields $\partial_s,\partial_t$, the c-projective vector field $v$ is given by
$$
v=\partial_{x_1}+\partial_{x_2}-((\beta+2\gamma)s+\alpha\, t)\partial_s-(\delta\,s+(2\beta+\gamma)t)\partial_t.
$$

It is straightforward to check that the lower-right block of \eqref{eq:LieA} is equivalent to the equations $\tfrac{\partial v^s}{\partial t}=-\alpha,\tfrac{\partial v^t}{\partial s}=-\delta,\tfrac{\partial v^t}{\partial t}-\tfrac{\partial v^s}{\partial s}=-(\beta-\gamma)$. Thus, the solution for $v$ from above also satisfies \eqref{eq:LieA}.

Now let us choose the metric $g$ appropriately within its c-projective class and choose an appropriate solution $A\in \mathcal{A}(g,J)$ such that, in the basis $h,\hat{h}$ of $\mathcal{S}([\gnabla],J)$ corresponding to $g,A$, one of the normal forms in \eqref{eq:normalformsv} holds for the endomorphism $\mathcal{L}_v:\mathcal{S}([\gnabla],J)\rightarrow \mathcal{S}([\gnabla],J)$. We obtain the following:
\begin{itemize}
\item Let $\left(\begin{array}{cc}
\gamma&\alpha\\\delta&\beta
\end{array}\right)=\left(\begin{array}{cc}
0&1\\0&0
\end{array}\right).$ Then the \textsc{ode} system \eqref{eq:ODEsLiouville} takes the form
$$
\frac{\partial\rho}{\partial x_1}=1,\,\,\,\frac{\partial\sigma}{\partial x_2}=1,\,\,\,\frac{\partial v^x}{\partial x_1}=\frac{\partial v^y}{\partial x_2}=0.
$$

Thus, the functions $\rho,\sigma,v^x,v^y$ are given by
$$
\rho(x_1)=x_1,\,\,\,\sigma(x_2)=x_2,\,\,\,v^x=c,\,\,\,v^y=d,
$$
where $c,d$ are non-zero constants.

Up to rescaling and up to adding constant linear combinations of the Killing vector fields $\partial_s,\partial_t$, the c-projective vector field $v$ is given by
$$
v=\partial_{x_1}+\partial_{x_2}-t\partial_s.
$$

\item Let $\left(\begin{array}{cc}
\gamma&\alpha\\\delta&\beta
\end{array}\right)=\left(\begin{array}{cc}
1&0\\0&\beta
\end{array}\right)\mbox{ for }\beta\neq 1.$ Then the \textsc{ode} system \eqref{eq:ODEsLiouville} takes the form
$$
\frac{\partial\rho}{\partial x_1}=(\beta-1)\rho,\frac{\partial\sigma}{\partial x_2}=(\beta-1)\sigma,\frac{2}{v^x}\frac{\partial v^x}{\partial x_1}=-\beta-2,\frac{2}{v^y}\frac{\partial v^y}{\partial x_2}=-\beta-2.
$$

Thus, $\sigma,\rho,v^x,v^y$ are given by
$$
\rho(x_1)=c\mathrm{e}^{(\beta-1)x_1},\,\,\,\sigma(x_2)=d\mathrm{e}^{(\beta-1)x_2},
$$
$$
v^x(x_1)=c'\mathrm{e}^{-\tfrac{1}{2}(\beta+2)x_1},\,\,\,v^y(x_2)=d'\mathrm{e}^{-\tfrac{1}{2}(\beta+2)x_2}.
$$

Up to rescaling and up to adding constant linear combinations of the Killing vector fields $\partial_s,\partial_t$, the c-projective vector field $v$ is given by
$$
v=\partial_{x_1}+\partial_{x_2}-(\beta+2)s\,\partial_s-(2\beta+1)t\,\partial_t.
$$

\item Let $\left(\begin{array}{cc}
\gamma&\alpha\\\delta&\beta
\end{array}\right)=\left(\begin{array}{cc}
1&1\\0&1\end{array}\right).$ Then, the \textsc{ode} system  \eqref{eq:ODEsLiouville} takes the form
$$
\frac{\partial\rho}{\partial x_1}=1,\frac{\partial\sigma}{\partial x_2}=1,\frac{2}{v^x}\frac{\partial v^x}{\partial x_1}=-3,\frac{2}{v^y}\frac{\partial v^y}{\partial x_2}=-3.
$$

Thus, $\rho,\sigma,v^x,v^y$ are given by
$$
\rho(x_1)=x_1,\,\,\,\sigma(x_2)=x_2,\,\,\,v^x(x_1)=ce^{-\tfrac{3}{2}x_1},\,\,\,v^y(x_2)=de^{-\tfrac{3}{2}x_2}.
$$

Up to rescaling and up to adding constant linear combinations of the Killing vector fields $\partial_s,\partial_t$, the c-projective vector field $v$ is given by
$$
v=\partial_{x_1}+\partial_{x_2}-(3s+t)\partial_s-3t\,\partial_t.
$$

\item Let $\left(\begin{array}{cc}
\gamma&\alpha\\\delta&\beta
\end{array}\right)=\left(\begin{array}{cc}
\beta&-1\\1&\beta
\end{array}\right).$ Then, the \textsc{ode} system \eqref{eq:ODEsLiouville} takes the form
$$
\frac{\partial\rho}{\partial x_1}=- \rho^2-1,\frac{\partial\sigma}{\partial x_2}=-\sigma^2-1,\frac{2}{v^x}\frac{\partial v^x}{\partial x_1}=-\rho-3\beta,\frac{2}{v^y}\frac{\partial v^y}{\partial x_2}=-\sigma-3\beta.
$$

Thus, $\rho,\sigma,v^x,v^y$ are given by
$$
\rho(x_1)=-\tan(x_1),\,\,\,\sigma(x_2)=-\tan(x_2),
$$
$$
v^x(x_1)=\frac{c\mathrm{e}^{-\tfrac{3}{2}\beta x_1}}{\sqrt{|\cos(x_1)|}},\,\,\,v^y(x_2)=\frac{d\mathrm{e}^{-\tfrac{3}{2}\beta x_2}}{\sqrt{|\cos(x_2)|}}.
$$

Up to rescaling and up to adding constant linear combinations of the Killing vector fields $\partial_s,\partial_t$, the c-projective vector field $v$ is given by
$$
v=\partial_{x_1}+\partial_{x_2}-(3\beta\, s-t)\partial_s-(s+3\beta\, t)\partial_t.
$$

\end{itemize}

\subsubsection{The complex Liouville case: Solving the \textsc{pde} systems \eqref{eq:PDEvgJ} and \eqref{eq:LieA}}

Let us work in the complex Liouville coordinates \eqref{eq:normalform2},\eqref{eq:normalformA2} and consider first the upper-left block of the \textsc{pde} systems \eqref{eq:PDEvgJ} and \eqref{eq:LieA}.

By direct calculation, we obtain that the upper-left block of \eqref{eq:PDEvgJ} is equivalent to the \textsc{pde}s
\begin{align}
\begin{array}{c}
\bar{v}^z\frac{\partial\bar{\rho}}{\partial \bar{z}}-v^z\frac{\partial \rho}{\partial z}+2(\bar{\rho}-\rho)\frac{\partial v^z}{\partial z}=-(\bar{\rho}-\rho)(\delta \rho+\delta (\rho+\bar{\rho})+3\gamma),\vspace{1mm}\\
\frac{\partial\bar{v}^v}{\partial z}-\frac{\partial v^z}{\partial\bar{z}}=0.
\end{array}\label{eq:ULBcplxLiouville1}
\end{align}

Similarly, we obtain that the upper-left block of \eqref{eq:LieA} is equivalent to the \textsc{pde}s
\begin{align}
\begin{array}{c}
v^z\frac{\partial \rho}{\partial z}=-\delta \rho^2+(\beta-\gamma)\rho+\delta,\quad\frac{\partial v^z}{\partial \bar{z}}=0.
\end{array}\label{eq:ULBcplxLiouville2}
\end{align}

In particular, we see that $v^z$ is a holomorphic function. Using \eqref{eq:ULBcplxLiouville2}, we can simplify \eqref{eq:ULBcplxLiouville1} and finally are left with the two equations
\begin{align}
\begin{array}{c}
v^z\frac{\partial \rho}{\partial z}=-\delta \rho^2+(\beta-\gamma)\rho+\alpha,\quad 2\frac{\partial v^z}{\partial z}=-\delta \rho-\beta-2\gamma.
\end{array}\label{eq:ULBcplxLiouville3}
\end{align}

As in the real case, it is now convenient to introduce new coordinates. Since the function $v^z$ is holomorphic, the $1$-form $\frac{1}{v^z}dz$ is closed and we can introduce a holomorphic change of coordinates by $dw=\frac{1}{v^z}dz$.

In this new coordinate the equations  \eqref{eq:ULBcplxLiouville3} take the form
\begin{align}
\begin{array}{c}
\frac{\partial \rho}{\partial w}=-\delta \rho^2+(\beta-\gamma)\rho+\alpha,\quad \frac{2}{v^z}\frac{\partial v^z}{\partial w}=-\delta \rho-\beta-2\gamma.
\end{array}\label{eq:ODEscplxLiouville}
\end{align}

Moreover, $v_D$ is given by $v_D=\partial_{w}+\partial_{\bar{w}}=\partial_{x_1},$ where $w=x_1+\mathrm{i}\, x_2$.

Similar to the case of real Liouville coordinates, we can solve the lower-right block of \eqref{eq:PDEvgJ} with respect to the unknown functions $v^s(s,t),v^t(s,t)$ by replacing the derivatives of $\rho,v^z$ by the equations \eqref{eq:ODEscplxLiouville}. By a straightforward calculation, we obtain exactly the same equations \eqref{eq:LHSPDE},\eqref{eq:RHSPDE} (with $\rho+\sigma$ replaced by $\rho+\bar{\rho}$) as in the real case. Thus, up to adding constant linear combinations of the Killing vector fields $\partial_s,\partial_t$, the c-projective vector field $v$ is given by
$$
v=\partial_w+\partial_{\bar{w}}-((\beta+2\gamma)s+\alpha\, t)\partial_s-(\delta\,s+(2\beta+\gamma)t)\partial_t.
$$

It is straightforward to check that this vector field also solves the equation $\eqref{eq:LieA}$.

Let us now solve the system \eqref{eq:ODEscplxLiouville} on the functions $\rho,v^z$ in each of the cases \eqref{eq:normalformsv} for the normal forms of $\mathcal{L}_v$.
\begin{itemize}
\item Let $\left(\begin{array}{cc}
\gamma&\alpha\\\delta&\beta
\end{array}\right)=\left(\begin{array}{cc}
0&1\\0&0
\end{array}\right).$ Then the \textsc{ode} systems \eqref{eq:ODEscplxLiouville} become equal to the equations
$$
\frac{\partial \rho}{\partial w}=1,\,\,\,\frac{2}{v^z}\frac{\partial v^z}{\partial w}=0.
$$
Thus, the functions $h,v^z$ are given by
$$
\rho(w)=w,\,\,\,v^z(w)=c+\mathrm{i}\,d
$$
for a non-zero constant $c+\mathrm{i}\, d$.
Up to rescaling and up to adding constant linear combinations of the Killing vector fields $\partial_s,\partial_t$, the c-projective vector field $v$ is given by
$$
v=\partial_w+\partial_{\bar{w}}- t\,\partial_s.
$$

\item Let $\left(\begin{array}{cc}
\gamma&\alpha\\\delta&\beta
\end{array}\right)=\left(\begin{array}{cc}
1&0\\0&\beta
\end{array}\right)\mbox{ for }\beta\neq 1.$ Then the \textsc{ode} system \eqref{eq:ODEscplxLiouville} takes the form
$$
\frac{\partial \rho}{\partial w}=(\beta-1)\rho,\,\,\,\frac{2}{v^z}\frac{\partial v^z}{\partial w}=-\beta-2.
$$

Thus, the functions $\rho,v^z$ are given by
$$
\rho(w)=e^{(\beta-1)w},\,\,\,v^z(w)=(c+\mathrm{i}\, d)e^{-\frac{1}{2}(\beta+2)w}.
$$
Up to rescaling and up to adding constant linear combinations of the Killing vector fields $\partial_s,\partial_t$, the c-projective vector field $v$ is given by
$$
v=\partial_w+\partial_{\bar{w}}-(\beta+2)s\,\partial_s-(2\beta+1)t\,\partial_t.
$$

\item Let $\left(\begin{array}{cc}
\gamma&\alpha\\\delta&\beta
\end{array}\right)=\left(\begin{array}{cc}
1&1\\0&1\end{array}\right).$ Then the \textsc{ode} system \eqref{eq:ODEscplxLiouville} takes the form
$$
\frac{\partial \rho}{\partial w}=1,\,\,\,\frac{2}{v^z}\frac{\partial v^z}{\partial w}=-3.
$$
Thus, the functions $\rho,v^z$ are given by
$$
\rho(w)=w,\,\,\,v^z(w)=(c+\i d)\mathrm{e}^{-\frac{3}{2}w}.
$$
Up to rescaling and up to adding constant linear combinations of the Killing vector fields $\partial_s,\partial_t$, the c-projective vector field $v$ is given by
$$
v=\partial_w+\partial_{\bar{w}}-(3s+t)\partial_s-3t\,\partial_t.
$$

\item Let $\left(\begin{array}{cc}           
\gamma&\alpha\\\delta&\beta
\end{array}\right)=\left(\begin{array}{cc}
\beta&-1\\1&\beta
\end{array}\right).$ Then the \textsc{ode} system \eqref{eq:ODEscplxLiouville} takes the form
$$
\frac{\partial \rho}{\partial w}=-\rho^2-1,\,\,\,\frac{2}{v^z}\frac{\partial v^z}{\partial w}=- \rho-3\beta.
$$
Thus, the functions $\rho,v^z$ are given by
$$
\rho(w)=-\tan(w),\,\,\,v^z(w)=\frac{(c+\i d)\mathrm{e}^{-\tfrac{3}{2}\beta w}}{\sqrt{\cos(w)}}.
$$
Up to rescaling and up to adding constant linear combinations of the Killing vector fields $\partial_s,\partial_t$, the c-projective vector field $v$ is given by
$$
v=\partial_w+\partial_{\bar{w}}-(3\beta\, s- t)\partial_s-(s+3\beta\,t)\partial_t.
$$

\end{itemize}


\subsection{The case when $A$ has order one}\label{sec:const}


Suppose that $A\in \mathcal{A}(g,J)$ has a constant eigenvalue $c$ and a non-constant eigenvalue $\rho$. In this section, we solve the \textsc{pde} systems \eqref{eq:PDEvgJ} and \eqref{eq:LieA} for the normal forms \eqref{eq:normalform3},\eqref{eq:normalformA3}.

Let us introduce a unitary coframing $\eta^1,\eta^2$ on $(\Sigma,h,j)$ such that $dx,\theta,\eta^1,\eta^2$ is a coframing on $M$. Note that we have
$$
\Omega=h(j.,.)=\eta^1\wedge \eta^2.
$$

Let us introduce functions $h^i$ on $\Sigma$ by $d \eta^i=h^i\eta^1\wedge \eta^2.$ Then, for the dual frame $\partial_x,\partial_t,\eta_1,\eta_2$ we have the relation
$$
[\eta_1,\eta_2]=\partial_t-h^1\eta_1-h^2\eta_2,
$$
all the other Lie bracket relations being zero.

Let us write the c-projective vector field into the form
\begin{align}
v=v^x\partial_x+v^t\partial_t+v^1\eta_1+v^2\eta_2.\label{eq:decompv}
\end{align}

We first evaluate the \textsc{pde} system \eqref{eq:LieA}. Note that in the frame $\partial_x,\partial_t,\eta_1,\eta_2$, the endomorphism $A$ is diagonal:
$$
A=\rho(\partial_x \otimes dx+\partial_t \otimes \theta)+c\left(\sum_{i=1}^2 \eta_i\otimes \eta^i\right).
$$

A straightforward calculation yields that \eqref{eq:LieA} is equivalent to the equations
\begin{align}
\begin{array}{c}
v^x\frac{\partial\rho}{\partial x}=-\delta\rho^2+(\beta-\gamma)\rho+\alpha,\quad\eta_1(v^x)=\eta_2(v^x)=0,\quad\eta_1(v^t)+ v^2=0,\vspace{1mm}\\
\eta_2(v^t)- v^1=0,\quad\frac{\partial v^1}{\partial x}=\frac{\partial v^1}{\partial t}=\frac{\partial v^2}{\partial x}=\frac{\partial v^2}{\partial t}=0,\quad
-\delta c^2+(\beta-\gamma)c+\alpha=0.
\end{array}\label{eq:LieAconstantev}
\end{align}

In particular, we see that the vector field $v_{\Sigma}=v^1\eta_1+v^2\eta_2$ is indeed (the lift to the horizontal distribution $\mathcal{H}=\mathrm{span}\{\partial_x,\partial_t\}^\perp$ of) a vector field on $\Sigma$. Moreover, we can apply $d \rho$ to the equation
$$
[\partial_t,v]=\frac{\partial v^x}{\partial t}\partial_x+\frac{\partial v^t}{\partial t}\partial_t
$$
which gives
$$
\frac{\partial v^x}{\partial t}\rho'=\partial_t (v(\rho))=\partial_t (-\delta\rho^2+(\beta-\gamma)\rho+\alpha)=0.
$$

Thus, $v^x$ is a function of $x$ alone. From this, since $v$ is holomorphic and $\partial_x$ is proportional to $J\partial_t$, we also obtain that
$$
[\partial_x,v]=\frac{\partial v^x}{\partial x}\partial_x,
$$
implying that $v^t$ does not depend on $x$.

As in the preceding sections, we introduce a new coordinate $x_1$ by $dx_1=\frac{1}{v^x}dx$. In particular, we have $v=\partial_{x_1}+v^t\partial_t+v_{\Sigma}$.

Let us now evaluate equation \eqref{eq:PDEvgJ}. It is straightforward to see that \eqref{eq:PDEvgJ} is equivalent to the equations
\begin{align}
\begin{array}{c}
\frac{2}{v^x}\frac{\partial v^x}{\partial x_1}=-(\delta\rho+\beta+2\gamma),\quad\frac{\partial v^t}{\partial t}=-(\delta c+\beta+2\gamma),\vspace{1mm}\\
\mathcal{L}_{v_{\Sigma}}h=-(\delta c+\beta+2\gamma)h.
\end{array}\label{eq:PDEvgJconstantev}
\end{align}

We see that the vector field $v_{\Sigma}$ is a homothety for  the metric $h$ on the base $\Sigma$. Therefore, in certain coordinates $u_1,u_2$, we have
\begin{align}
v_\Sigma=\partial_{u_1}, \quad h=e^{\lambda u_1}G(u_2)(du_1^2+du_2^2),\quad \Omega=e^{\lambda u_1}G(u_2)du_1\wedge du_2\label{eq:solonbase}
\end{align}
for an arbitrary function $G(u_2)$, where  the constant $\lambda$ is defined as
$$
\lambda=-(\delta c+\beta+2\gamma).
$$

Note that we can write $(jv_{\Sigma})^b=-v^2\eta^1+v^1\eta^2$ independently of the chosen frame $\eta_1,\eta_2$ on $\Sigma$. By \eqref{eq:LieAconstantev},\eqref{eq:PDEvgJconstantev}, given $h,\Omega$ and $v_{\Sigma}$ as in \eqref{eq:solonbase}, it remains to solve the equations
\begin{align}
\begin{array}{c}
\frac{\partial\rho}{\partial x_1}=-\delta\rho^2+(\beta-\gamma)\rho+\alpha,\quad\frac{2}{v^x}\frac{\partial v^x}{\partial x_1}=-(\delta\rho+\beta+2\gamma),\vspace{1mm}\\
dv^t=\lambda\theta+(jv_{\Sigma})^b,\quad d\tau=\Omega,\vspace{1mm}\\
-\delta c^2+(\beta-\gamma)c+\alpha=0
\end{array}\label{eq:eqconstnonconst}
\end{align}
on $\rho,v^x,v^t$ and $\tau$. We will solve these equations in each of the cases for the normal forms of $\mathcal{L}_v$ as they appear in Lemma \ref{lem:order1}. Note that in the first and third case, we will use the transpose matrix, since by Lemma \ref{lem:order1}, the first basis vector $h$ is degenerate but the second basis vector $\hat{h}$ is not and hence corresponds to a metric.
\begin{itemize}
\item Let $\left(\begin{array}{cc}
\gamma&\alpha\\\delta&\beta
\end{array}\right)=\left(\begin{array}{cc}
0&0\\1&0
\end{array}\right).$ Then, $c=0$ and $\lambda=0$, thus the formulas in \eqref{eq:solonbase} become
\begin{align}
v_\Sigma=\partial_{u_1}, \quad h=G(u_2)(du_1^2+du_2^2), \quad \Omega=G(u_2)du_1\wedge du_2.\nonumber
\end{align}
and \eqref{eq:eqconstnonconst} reads
\begin{align}
\begin{array}{c}
\frac{\partial\rho}{\partial x_1}=-\rho^2, \quad \frac{2}{v^x}\frac{\partial v^x}{\partial x_1}=-\rho, \quad d v^t=G(u_2)du_2, \quad d\tau=G(u_2)du_1\wedge du_2.
\end{array}\nonumber
\end{align}

The solutions for $\rho,v^x$ are
$$
\rho(x_1)=\frac{1}{x_1}, \quad  v^x (x_1)=\frac{c_1}{\sqrt{|x|}}.
$$

Further, we can set $\tau=u_1 G(u_2)du_2$ which solves $d \tau=\Omega$. By a change of variables $d\tilde{u}_2=G(u_2)du_2$, we have
$$
\tau=u_1d\tilde{u}_2, \quad h=G(\tilde{u}_2)du_1^2+\frac{1}{G(\tilde{u}_2)}d\tilde{u}_2^2,  \quad d v^t=d \tilde{u}_2.
$$

Finally, the solution for $v$ (up to rescaling and adding constant multiples of the Killing vector field $\partial_t$) is
$$
v=\partial_{x_1}+\tilde{u}_2\partial_t+\partial_{u_1}.
$$

\item Let $\left(\begin{array}{cc}
\gamma&\alpha\\\delta&\beta
\end{array}\right)=\left(\begin{array}{cc}
1&0\\0&\beta
\end{array}\right)\mbox{ for }\beta\neq 1.$ Then $c=0$ and $\lambda=-(\beta+2)$, thus the formulas in \eqref{eq:solonbase} become
\begin{align}
v_\Sigma=\partial_{u_1},h=e^{-(\beta+2)u_1}G(u_2)(du_1^2+du_2^2),\Omega=e^{-(\beta+2)u_1}G(u_2)du_1\wedge du_2.\nonumber
\end{align}
and \eqref{eq:eqconstnonconst} takes the form
\begin{align}
\begin{array}{c}
\frac{\partial\rho}{\partial x_1}=(\beta-1)\rho,\quad \frac{2}{v^x}\frac{\partial v^x}{\partial x_1}=-(\beta+2),\vspace{1mm}\\
d v^t=-(\beta+2)(dt-\tau)+e^{-(\beta+2)u_1}G(u_2)du_2,\vspace{1mm}\\
d\tau=e^{-(\beta+2)u_1}G(u_2)du_1\wedge du_2.
\end{array}\nonumber
\end{align}
The solutions for $\rho,v^x$ are
$$
\rho(x_1)=c_1e^{(\beta-1)x_1}, \quad v^x(x_1)=d_1 e^{-\frac{1}{2}(\beta+2)x_1}.
$$
Further, we have
\begin{itemize}
\item Subcase $\beta+2=0$: we introduce a new variable $d\tilde{u}_2=G(u_2)d u_2$. Then,
$$
\tau=u_1d\tilde{u}_2, \quad h=G(\tilde{u}_2)du_1^2+\frac{1}{G(\tilde{u}_2)}d\tilde{u}_2^2, \quad d v^t=d \tilde{u}_2.
$$

Finally, the solution for $v$ (up to rescaling and adding constant multiples of the Killing vector field $\partial_t$) is
$$
v=\partial_{x_1}+\tilde{u}_2\partial_t+\partial_{u_1}.
$$

\item Subcase $\beta+2\neq 0$: we can choose
$$
\tau=-\tfrac{1}{\beta+2}e^{-(\beta+2)u_1}G(u_2) du_2
$$
and obtain $d v^t=-(\beta+2)dt$. Finally, the solution for $v$ (up to rescaling and adding constant multiples of the Killing vector field $\partial_t$) is
$$
v=\partial_{x_1}-(\beta+2)t\partial_t+\partial_{u_1}.
$$

\end{itemize}

\item Let $\left(\begin{array}{cc}
\gamma&\alpha\\\delta&\beta
\end{array}\right)=\left(\begin{array}{cc}
1&0\\1&1\end{array}\right).$ Then $c=0$ and $\lambda=-3$, thus the formulas in \eqref{eq:solonbase} become
\begin{align}
v_\Sigma=\partial_{u_1}, \quad h=e^{-3u_1}G(u_2)(du_1^2+du_2^2), \quad \Omega=e^{-3 u_1}G(u_2)du_1\wedge du_2.\nonumber
\end{align}

and \eqref{eq:eqconstnonconst} takes the form
\begin{align}
\begin{array}{c}
\frac{\partial\rho}{\partial x_1}=-\rho^2, \quad \frac{2}{v^x}\frac{\partial v^x}{\partial x_1}=-\rho-3,\vspace{1mm}\\
d v^t=-3(dt-\tau)+e^{-3 u_1}G(u_2)du_2, \quad d\tau=e^{-3u_1 }G(u_2)du_1\wedge du_2.
\end{array}\nonumber
\end{align}

The solutions for $\rho,v^x$ are
$$
\rho(x_1)=\frac{1}{x_1}, \quad v^x (x_1)=\frac{c_1 e^{-\frac{3}{2}x_1}}{\sqrt{|x_1|}}.
$$

Further, we can set $\tau=-\frac{1}{3}e^{-3 u_1 }G(u_2)du_2$ which solves $d \tau=\Omega$. From this we obtain $d v^t =-3dt$.

Finally, the solution for $v$ (up to rescaling and adding constant multiples of the Killing vector field $\partial_t$) is
$$
v=\partial_{x_1}-3t\partial_t+\partial_{u_1}.
$$

\end{itemize}

This finally completes the last part of Theorem \ref{thm:normalformwithv}.


\section{Proof of the Yano-Obata conjecture in the  $4$-dimensional  pseudo-Riemannian case}


In this section, we prove the Yano-Obata conjecture (Theorem \ref{YOconjecture}).


 \subsection{The complex Liouville case}\label{sect5}


First, we exclude the complex Liouville case from our further considerations (see Theorem~\ref{thm:normalformwithv}). Our goal is to prove

\begin{thm}\label{thm:closedcplxconjugate}
Let $(M,g,J)$ be a closed connected   K\"ahler  surface   and suppose $A\in \mathcal{A}(g,J)$ has a complex eigenvalue $\rho=\mathcal{R}+\mathrm{i}\mathcal{I}$, $\mathcal I\ne 0$,  at least at one point. Then, $A$ is parallel.
\end{thm}

Let $p_{0}\in M$ be a point such that $\mathcal{I}(p_{0})=\mathcal{I}_{\mbox{\tiny max}}=\mbox{max}\{\mathcal{I}(p):p\in M\}$. Such a point exists since $M$ is compact.

We will say that a point $p\in M$ is called \emph{regular},  if $p\in M^\circ$,  i.e.  $A$ has two distinct eigenvalues at this point and their differentials are not zero.  Otherwise $p$ will be called \emph{singular}. Note that the regular points form an open and dense subset of $M$.  This implies that only two cases are possible:

\begin{itemize}

 \item either $\rho$ and $\bar\rho$ are constant eigenvalues of $A$ on the whole $M$ and then, according to \eqref{eq:main}, $A$ is parallel,

 \item or $\rho(p)=\mathcal{R}(p)+i\mathcal{I}(p)$ is a smooth function in a sufficiently small  neighborhood $U(p_0)$ of $p_0$ with $d\rho\ne 0$ almost everywhere on $U(p_0)$.

 \end{itemize}

 We will show, by contradiction, that the second case is impossible and hence the statement follows.

The proof will be organised as follows. From now on, we consider a small geodesically convex neighborhood $U=U(p_0)$  where $\rho$ is smooth and $\mathcal{I}(p) > 0$ for all $p\in U$.
 Lemmas \ref{lem:p0sing} and \ref{+3}  show that
$p_0$ is a singular point and moreover all singular points with $d\rho = 0$ lie
on a  certain geodesic which we denote by $\gamma$. The case when $\gamma$ contains also regular points will be considered in Lemma \ref{+4}.  The case when all points of
$\gamma$ are singular  will be considered in Lemma \ref{+5}.  In the both cases we will come to a contradiction with the maximum principle for holomorphic functions.

\begin{lem}\label{lem:p0sing}
The point $p_0$ is singular.
\end{lem}

\begin{proof}
Since $d\mathcal{I}=0$ at $p_0$, we have $d\rho=d\bar{\rho}$.
 In view of  Lemma \ref{lem:eigenvaluegradients},  $\textrm{grad} (\rho)$ is an eigenvector of $A$ corresponding to $\rho$  and $\textrm{grad}(\bar{\rho})$  is an eigenvector of $A$ corresponding to $\bar \rho$ so they coincide if and only if  $d\rho=d\bar \rho=0$ at $p_0$, i.e., $p_0$ is singular.
\end{proof}

\begin{lem} \label{+3}
There exists a geodesic $\gamma:(a,b)\to  U$
such that all singular points of $U$ lie on $\gamma$.
\end{lem}

\begin{proof}
If  there are at most two singular points lying in $U$, there is nothing to prove.  Suppose there exist three singular points   $p_0,p_2,p_2\in U$   that do not lie on a geodesic contained in $U$. Then, for a point $q\in U$,
denote by $\gamma_i$, $i=0,1,2$,
 the geodesics lying in $U$ and connecting $p_i$ and $q$; we assume $\gamma_i(0)= p_i$  and $\gamma_i(1)= q$. For a generic point  $q\in U$, the velocity vectors $\dot\gamma_i(1)\in T_qM$ are linearly independent. On the other hand, they are  orthogonal to the Killing vector fields $K_1, K_2$. Indeed, the functions
 $$
 t\longmapsto g(K_1(\gamma_i(t)),\dot{\gamma}_i(t)) \quad \textrm{and} \quad  t\longmapsto g(K_2(\gamma_i(t)),\dot{\gamma}_i(t))
 $$
 are constant on the geodesics and vanish at $t=0$ since at this point $K_1= K_2=0$.
 We may assume that the point $q$ is regular (otherwise replace it by a regular point from a very small neighborhood).  Then, the linearly independent vectors $\dot\gamma_i \in T_q M$ are contained in the two-dimensional subspace given by the orthogonal complement to $\mathrm{span}\{K_1,K_2\}$.
  This gives us  a contradiction and the claim follows.
\end{proof}

The proof of Theorem \ref{thm:closedcplxconjugate} follows now from the next two lemmas below.

\begin{lem} \label{+4}
Consider the geodesic segment $\gamma:(a, b)\to U$  containing all singular points of $U$. Then $\gamma$ contains no regular points.
\end{lem}
\begin{figure}
  \includegraphics[scale=1.4]{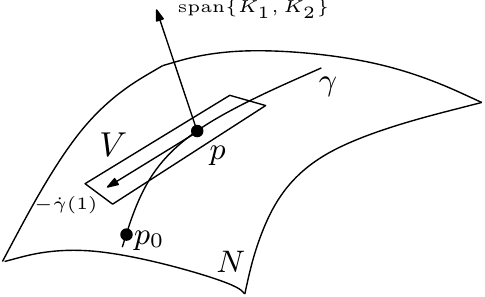}
  \caption{The construction of the submanifold $N$.}\label{1}
\end{figure}

\begin{proof}
Let $\gamma((a,b))$ be the geodesic line segment containing all singular points and $\gamma(0)= p_0$. Assume, by contradiction, that $p=\gamma(1)$ is regular.
Consider $T_pM$ and  the two-dimensional subspace $\mathrm{span}\{K_1,K_2\}^\perp\subset T_pM$.
The straight line  segment connecting the vectors  $0\in T_pM$ and   $-\dot \gamma(1) \in  T_pM$ lies in  $\mathrm{span}\{K_1,K_2\}^\perp$ since as we have explained in the proof of Lemma
\ref{+3}  the geodesic $\gamma$ is orthogonal to the Killing vector fields at every point.

We  denote by $V$ a thin tubular neighborhood of  this segment in $\mathrm{span}\{K_1,K_2\}^\perp\subset T_pM$, see figure \ref{1}. If the neighborhood is sufficiently thin,   the exponential mapping is well defined on $V$ and is an embedding of $V$ into $U$, so its image, which we denote by $N$, is a two-dimensional submanifold.

Notice that all the points $q\in N$ which do not belong to $\gamma$ are regular. On the other hand we know from the proof of Theorem \ref{thm:normalform}  that on $U\setminus\{\gamma((a,b))\}$ there is an integrable two-dimensional distribution defined by the subspaces $\mathrm{span}\{K_1,K_2\}^\perp$ whose leaves are totally geodesic. Hence every geodesic starting at $p$ in the direction of $V$ belongs to the leaf of this distribution through $p$ if it does not leave the set of regular points. In other words, the subset
$$
N'=\{\mathrm{exp}_p(X):X\in V, \mathrm{exp}_p(tX)\mbox{ regular for all }t\in [0,1]\}
$$
of $N$ is contained as an open subset in the totally geodesic integral leaf of the distribution $\mathrm{span}\{K_1,K_2\}$ through the point $p$ and it is open and dense in $N$, since $\dot{\gamma}(1)$ is the only direction in $V$ in which geodesics starting initially in this direction can meet singular points. This implies that $N$ is a totally geodesic submanifold since this condition holds in a
neighborhood of almost every point.

Furthermore, for all points $q\in N'$ the tangent space $T_qN$ coincides with $\mathrm{span}\{K_1,K_2\}^\perp$. Therefore $T_qN$  is $A$-invariant and the eigenvalues of the restriction $A_N=A|_{T_qN}$ are $\rho$ and $\bar \rho$ with multiplicity one.
By continuity, $T_qN$ is still $A$-invariant even for $q\in N\setminus N'$ and $\rho, \bar\rho$ are still the eigenvalues of the restriction $A_N$.

Thus, we see that at each point $q\in N$,  the operator $A_N$ is conjugate to the matrix $\begin{pmatrix} \mathcal R & \mathcal I \\ -\mathcal I & \mathcal R\end{pmatrix}$ where $\rho=\mathcal{R}+\mathrm{i}\,\mathcal{I}$. This allows us to introduce a natural complex structure $J_N$ on $N$ by setting $J_N = \frac{1}{\mathcal I}(A_N - \mathcal R\cdot \mathrm{Id})$.  Alternatively, we can define $J_N$ by noticing that (up to the sign) $J_N$ is the only complex structure on $N$ that commutes with $A_N$.

It is straightforward to see that at each regular point $q\in N'$ this complex structure coincides with the one given by the local classification of Theorem \ref{thm:normalform}. In particular, $\rho:N\rightarrow \mathbb{C}$ is a holomorphic function on $N'$ with respect to $J_N$. Since $\rho$ is a smooth function on $U$ and hence on $N$, we obtain that $\rho$ is holomorphic on the whole of $N$. This implies that
$\mathcal{I}:N\rightarrow \mathbb{R}$ is a harmonic function that attains its maximal value at $p_0\in N$.
Hence, by the maximum principle, $\mathcal{I}$ and therefore $\rho$ are constant on $N$. On the other hand, $\rho$ does not depend on the direction orthogonal to $N$ (see Corollaries \ref{cor:killing} and \ref{cor:commutekilling}). Thus, we obtain that $\rho$ is constant on an open neighborhood of $p_0$ in $M$, which contradicts to our assumption that $d\rho\ne 0$ almost everywhere on $U(p_0)$. This proves the Lemma.
\end{proof}

The next lemma excludes the other possibility.

\begin{lem} \label{+5}
Consider the geodesic segment $\gamma:(a, b)\to U$  containing all singular points of $U$. Then $\gamma$ cannot consist of singular points only.
\end{lem}

 \begin{figure}
    \includegraphics[scale=1.25]{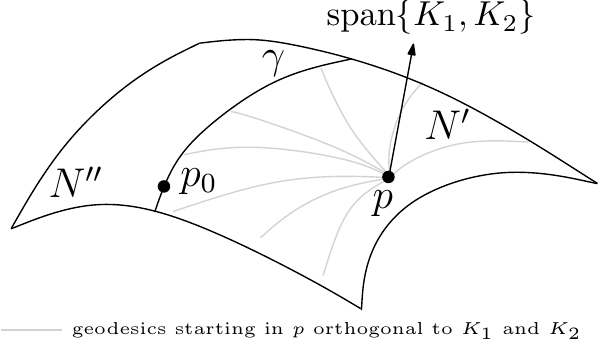}
 \caption{All points on $\gamma$ are singular.}\label{2}
        \end{figure}

\begin{proof}
By contradiction, assume that all points of $\gamma$ are singular.
Since the differential of $\rho$ vanishes at each singular point, $\rho$ is constant along $\gamma$. Let $p\in U$ be a regular point and let $V\subseteq T_p M$ be a star-like open neighborhood of $0$ such that it is mapped diffeomorphically onto $U$ by the exponential mapping. We define the two-dimensional submanifold $N$ to be the image of $V\cap \mathrm{span}\{K_1,K_2\}^\perp$ under the exponential mapping. The curve $\gamma$ lies on $N$ and divides it into two connected components, see figure \ref{2}. We denote by $N'$ the component containing $p$ joint with $\gamma((a,b))$. Consider the complex structure on $N'$ constructed similar as in the proof of Lemma \ref{+4} and smoothly extend it from $N'$ through the boundary $\gamma$ to the domain $N''$ so that, as the result, we obtain a complex structure on $N$. {More precisely, we have the field of operators $A_N=A|_{TN}$  defined on $N'$ as in Lemma \ref{+4}. This field $A_N$ is smooth on $N'$ including its boundary $\gamma$. We can smoothly extend this field from $N'$ to $N''$  in such a way that $A_N$ still has two distinct eigenvalues at every point, and then apply the formula $J_N = \frac{1}{\mathcal I}(A_N - \mathcal R\cdot \mathrm{Id})$ from Lemma \ref{+4} to define the complex structure.}

The function $\rho$ (restricted to $N'$) is holomorphic with respect to this complex structure and is constant on $\gamma$. Then, it is constant by Morera's theorem \cite{Morera} on $N'$. To come to a contradiction, we can now proceed completely analogous to the proceeding lemma and the claim follows.
\end{proof}

\subsection{The Liouville and degenerate cases}

It is our goal to prove the Yano-Obata conjecture Theorem \ref{YOconjecture}, that is, we want to show that the existence of an essential c-projective vector field $v$ on a closed connected pseudo-Riemannian K\"ahler surface $(M,g_0,J)$ implies that $g_0$ has constant holomorphic sectional curvature.

Having the normal forms of Theorem \ref{thm:normalformwithv} and Theorem \ref{thm:allmetricsintheclass} at hand, the proof of Theorem \ref{YOconjecture} is basically an exercise using your favourite computer algebra program. We will proceed by calculating the scalar curvature of the metrics $\hat{g}$, given by the formulas \eqref{eq:ghat1} and \eqref{eq:ghat3} in the Liouville case and degenerate case respectively. For certain values of the parameters $c,C$ in the formulas \eqref{eq:ghat1} and \eqref{eq:ghat3}, the metric $\hat g$ is the initial metric $g_0$. The condition that the restriction $\mathrm{Scal}(\gamma(\tau))$ of the scalar curvature to an integral curve $\gamma$ of the c-projective vector field $v$ cannot explode when $\tau$ approaches infinity, will provide restrictions on the parameters appearing in Theorem \ref{thm:normalformwithv}.

We will show that for the parameters such that the scalar curvature is bounded, the corresponding metrics have constant holomorphic sectional curvature. {Note that we do not calculate the scalar curvature of the metrics $g$ from Theorem \ref{thm:normalformwithv}, since for these metrics it follows immediately that the restriction $\rho(\gamma(\tau))$ of the nonconstant eigenvalue $\rho$ of the tensor $A\in \mathcal{A}(g,J)$ (given by the formulas from Theorem \ref{thm:normalform}) explodes if $\tau$ approaches infinity. These metrics cannot be globally defined and we work with the generic representatives \eqref{eq:ghat1} and \eqref{eq:ghat3} of the c-projective class of $g_0$ instead.}

Let us explain, that it is sufficient to consider the cases L2 and D2 in Theorem \ref{thm:normalformwithv}. Below we use the results and notation that have been introduced in the beginning of Section \ref{sec:solvinglie}. Let $v$ be an essential c-projective vector field. Without loss of generality, we can assume that $d(g,J)=2$ and thus, can introduce $h=h_g$ and $\hat{h}=-\mathcal{L}_v h$ as a basis in $\mathcal{S}([\gnabla],J)$ so that the matrix of $\mathcal{L}_v:\mathcal{S}([\gnabla],J)\rightarrow \mathcal{S}([\gnabla],J)$ in this basis becomes
\begin{align}
\mathcal{L}_v=\left(\begin{array}{cc}0&\alpha\\-1&\beta\end{array}\right)\label{eq:matrixofessentialv}
\end{align}
for certain constants $\alpha,\beta$. By Lemma \ref{key} and the equation \eqref{eq:LieA}, the eigenvalues $\rho$ of $A=\hat{h}h^{-1}$ satisfy $v(\rho)=\rho^2+\beta \rho +\alpha$ on the dense and open subset where they are smooth. In particular,  the restrictions $\rho(\tau)=\rho(\gamma(\tau))$ of the eigenvalues of $A$ to the integral curves $\gamma$ of $v$ satisfy the equation
\begin{align}
\frac{d\rho}{d\tau}=\rho^2+\beta \rho+\alpha.\label{eq:rhoalongv}
\end{align}

The behaviour of solutions of \eqref{eq:rhoalongv} depends on the constants $\alpha$ and $\beta$, more precisely, on the roots $\lambda_1, \lambda_2$ of the quadratic equation $\rho^2 + \beta \rho + \alpha=0$. On the other hand,  different types of the roots correspond to different cases in the classification Theorem \ref{thm:normalformwithv}. These relationships can be summarised as follows (in view of \S\ref{sect5} we do not include the complex Liouville case):
\begin{itemize}

\item $\lambda_1=\lambda_2=0$:  cases L1 and D1 in Theorem \ref{thm:normalformwithv}. The non-constant solutions of \eqref{eq:rhoalongv} are all unbounded,

\item $\lambda_1=\lambda_2\ne 0$:  cases L3 and D3 in Theorem \ref{thm:normalformwithv}. The non-constant solutions of \eqref{eq:rhoalongv} are all unbounded,

\item $\lambda_1,\lambda_2 \in \R$ and $\lambda_1\ne\lambda_2$: cases L2 and D2 in Theorem \ref{thm:normalformwithv}. There exist non-constant bounded solutions of \eqref{eq:rhoalongv} of the following form
\begin{align}
\rho(\tau)=-\tfrac{\beta}{2}-\sqrt{c}\,\mathrm{tanh}(\sqrt{c}(\tau+d)),\label{eq:rhot}
\end{align}
where $c=\tfrac{\beta^2}{4}-\alpha$ is necessarily positive and $d$ is some integration constant,

\item $\lambda_1$ and $\lambda_2$ are complex conjugate and $\lambda_1\ne\lambda_2$: case L4 in Theorem \ref{thm:normalformwithv}. The solutions of \eqref{eq:rhoalongv} are all unbounded.
\end{itemize}

This implies, that we can restrict to the cases L2 and D2 since only in these two cases the (non-constant) eigenvalues of $A$ might be bounded.

As mentioned in the beginning of this section, we will evaluate the scalar curvature of the metrics $\eqref{eq:ghat1}$ and $ \eqref{eq:ghat3}$ by using the coordinates provided in the cases L2 and D2 from Theorem \ref{thm:normalformwithv} respectively. These coordinates can be introduced in a neighborhood of every point of $M^\circ,$ the dense and open subset of $M$, where the eigenvalues $\rho,\sigma$ of $A$ satisfy $\rho\neq \sigma$ and $d\rho\neq0,d\sigma\neq 0$ unless they are constant.

To evaluate the behaviour of $\mathrm{Scal}(\gamma(\tau))$ along the integral curves $\gamma$ of $v$ by using the coordinates in Theorem \ref{thm:normalformwithv}, we have to show that these integral curves do not leave $M^\circ$.

\begin{lem}
The flow $\Phi^v_\tau$ of $v$ leaves $M^\circ$ invariant.
\end{lem}

\begin{proof} 
This follows immediately from Corollary \ref{cor:bols1}.
\end{proof}

We also need to explain why we are allowed to use {\it local}  coordinates  from Theorem \ref{thm:normalformwithv} to study the {\it global} behaviour of the function $\mathrm{Scal}(\gamma(\tau))$.  The situation we are dealing with can be described as follows.  We have a K\"ahler surface $(M,g,J)$ with a c-projective vector field $v$ and  a canonical model given by Theorem \ref{thm:normalformwithv}. We know that locally these two geometric objects are isomorphic, i.e., there exists a local isometry between them sending one c-projective vector field to the other. In order to be sure that the evolution of $\mathrm{Scal}$ along the corresponding flows is identical for these two manifolds,  we need to show that this local isometry can be prolonged along a trajectory of $v$ as long as we wish.  Let us prove this fact.   

Let $(M, g, J)$  and $(M',g',J')$ be two K\"ahler manifolds with c-projective vector fields $v$ on $M$ and $v'$ on $M'$. Consider two points $p\in M$ and $p'\in M'$ and the trajectories $\gamma(\tau)$ and $\gamma'(\tau)$  of $v$ and $v'$ respectively such that $p=\gamma(0)$ and $p'=\gamma'(0)$.
Assume that the both trajectories are defined for $\tau\in [0, T]$ and  there is a local isometry $F_0:  U(p) \to U'(p')$ between some neighbourhoods of these points which sends $v$ to $v'$, i.e. $dF(v)=v'$.  Then using the flows $\Phi^v_\tau$ and $\Phi^{v'}_\tau$ we can define a local isomorphism $F_\tau : V(q) \to V'(q')$ between some neighbourhoods of $q=\gamma(\tau)$ and $q'=\gamma'(\tau)$,  $\tau\in [0,T]$ by putting:
$$
F_\tau = \Phi^{v'}_\tau \circ F_0 \circ {\Phi^{v}_{-\tau}}.
$$

\begin{lem}
\label{lem:bols2} 
 $F_\tau$ is a local isometry. 
\end{lem}

\begin{proof} Without loss of generality we can assume that the flows  $\Phi^{v}_\tau$ and $\Phi^{v'}_\tau$ are well defined for $\tau\in [0,T]$ on the neighbourhoods $U(p)$ and $U'(p')$ respectively (otherwise we just take some smaller neighbourhoods).   Consider the ``orbits'' of these neighbourhoods under the action of these ``partial'' flows
$$
U_T (p) = \bigcup_{\tau\in [0,T]} \Phi^{v}_\tau (U(p))\quad \mbox{and} \quad  U'_T (p') = \bigcup_{\tau\in [0,T]} \Phi^{v'}_\tau (U'(p')).
$$

For simplicity,  we may assume that $\gamma: [0,T]\to M$ is an embedding (i.e., $\gamma$ is not closed) and $U(p)$ is sufficiently small so that $U_T(p)$ is a regular tubular neighbourhood of $\gamma([0,T])$.
It is a standard fact in the theory of dynamical systems that all the maps $F_\tau$ agree and can be glued together into a single map $F: U_T (p)  \to U'_T (p')$ that can be though of as a prolongation of $F_0$ along the flow(s). So actually we are going to prove that $F$ is a local isometry (we still say ``local'',  as in general $F$ is not necessarily one-to-one,  this map can behave as a covering).  However the property we need to verify is local,  so without loss of generality we may assume that $F$ is a diffeomorphism.

Since $F$ is a composition of three c-projective maps (two flows of c-projective vector fields and one isometry),  $F$ is a c-projective map itself.  By taking the pullback of the K\"ahler structure $(g', J')$ from   $U'_T (p')$  to $U_T (p)$  we obtain a K\"ahler structure on $U_T (p)$  that is  c-projectively equivalent to $(g,J)$ and coincides with it on $U(p)$.  Since equations \eqref{eq:main} are of finite type, then the family of such K\"ahler structures has finite dimension and, moreover, each such structure is defined by finitely many initial conditions at one point, so  we conclude that $(g,J)$ coincides with the $F$-pullback of $(g', J')$ on the whole of $U_T(p)$ and therefore $F$ is an isometry. \end{proof}

After these preliminaries, we can give the proof of the Yano-Obata conjecture in the Liouville case L2 and degenerate case D2.

\subsubsection{The Liouville case L2}
Let us calculate the scalar curvature $\mathrm{Scal}$ of the metric  {$\hat{g}$ given by formula \eqref{eq:ghat1} in the local Liouville coordinates L2 from Theorem \ref{thm:normalformwithv}} and restrict it to the flow lines of the c-projective vector field $v$.  By Theorem \ref{thm:normalformwithv}, the evolution of the coordinates $(x,y)$ along the flow $\Phi^v_\tau$ of $v$ is given by $\Phi^v_\tau(x,y)=(x+\tau,y+\tau)$. Starting from an arbitrary point $p\in M^\circ$, we can assume without loss of generality that $p$ has first coordinates equal to $x=0,y=0$.   Note that this necessarily implies  {that the constants $c_1,c_2$ from the case L2 of Theorem \ref{thm:normalformwithv} satisfy $c_1-c_2\neq 0$.  
  After having calculated $\mathrm{Scal}$, the restriction $\mathrm{Scal}(\tau)=\mathrm{Scal}(\Phi^v_\tau(p))$ is given by replacing $x$ and $y$ in $\mathrm{Scal}$ with $\tau$ since $\mathrm{Scal}$ does not depend on the coordinates $s,t$ and their evolution along the integral curves of $v$.} We obtain 
\begin{align}
\mathrm{Scal}(\tau)=\frac{3\,e^{3\,\tau}}{d_1^2d_2^2(c_1-c_2)}\cdot\frac{C_0c^7+C_1c^6e^{(\beta-1)\tau}+...+C_2c^5e^{2(\beta-1)\tau}+...+C_7e^{7(\beta-1)\tau}}{D_0c^4+D_1c^3e^{(\beta-1)\tau}+...+D_4 e^{4(\beta-1)\tau}}\label{eq:scalalongv}
\end{align}
where the constants $C_0,...,C_7$ and $D_0,...,D_7$ are given by
\begin{align}
\begin{array}{ll}
C_0=&2\beta(\beta+\tfrac{1}{2})(d_1^2-d_2^2),\vspace{1mm}\\
C_1=&6(c_2d_2^2-c_1d_1^2)(\beta+\tfrac{1}{2})\beta+4(c_1d_2^2-c_2d_1^2)(\beta+2)\beta,\vspace{1mm}\\
C_2=&6(c_1^2d_1^2-c_2^2d_2^2)(\beta+\tfrac{1}{2})\beta+2(c_2^2d_1^2-c_1^2d_2^2)(1+7\beta+\beta^2)+12c_1c_2(d_1^2-d_2^2)(\beta+2)\beta,\vspace{1mm}\\
C_3=&2(c_2^3d_2^2-c_1^3d_1^2)(\beta+\tfrac{1}{2})\beta+8(c_1^3d_2^2-c_2^3d_1^2)(\beta+\tfrac{1}{2})+6c_1c_2(c_1d_2^2-c_2d_1^2)(1+7\beta+\beta^2)\vspace{1mm}\\&+12c_1c_2(c_2d_2^2-c_1d_1^2)(\beta+2)\beta,\vspace{1mm}\\
C_4=&(c_2^4d_1^2-c_1^4d_2^2)(\beta+2)+4c_1c_2(c_1^2d_1^2-c_2^2d_2^2)(\beta+2)\beta+6c_1^2c_2^2(d_1^2-d_2^2)(1+7\beta+\beta^2)\vspace{1mm}\\&+24c_1c_2(c_2^2d_1^2-c_1^2d_2^2)( \beta+\tfrac{1}{2}),\vspace{1mm}\\
C_5=&3c_1c_2(c_1^3d_2^2-c_2^3d_1^2)(\beta+2)+24c_1^2c_2^2(c_1d_2^2-c_2d_1^2)(\beta+\tfrac{1}{2})\vspace{1mm}\\&+2c_1^2c_2^2(c_2d_2^2-c_1d_1^2)(1+7\beta+\beta^2),\vspace{1mm}\\
C_6=&3c_1^2c_2^2(c_2^2d_1^2-c_1^2d_2^2)(\beta+2)+8c_1^3c_2^3(d_1^2-d_2^2)(\beta+\tfrac{1}{2}),\vspace{1mm}\\
C_7=&c_1^3c_2^3(c_1d_2^2-c_2d_1^2)(\beta+2)
\end{array}\label{eq:scalconstantsC}
\end{align}
and
$$
D_0=1,D_1=-2(c_1+c_2),D_2=2c_1c_2+(c_1+c_2)^2,D_3=-2c_1c_2(c_1+c_2),D_4=c_1^2c_2^2.
$$
In the case  $\beta-1>0$, we see from \eqref{eq:scalalongv} that the condition $|\lim_{\tau\to\infty}\mathrm{Scal}|<\infty$ implies 
$$
C_4=C_5=C_6=C_7=0.
$$
From \eqref{eq:scalconstantsC}, we see that there is no choice of $c_1,c_2,d_1,d_2,\beta$ such that this condition is fulfilled.

Let us suppose that $\beta-1<0$. From $|\lim_{\tau\to\infty}\mathrm{Scal}|<\infty$ it follows that $C_0=0$ that is, we either have $\beta=0$ or $\beta=-\tfrac{1}{2}$ or $d_1^2-d_2^2=0$.

\emph{Case $\beta=0$:} Note that the condition $|\lim_{\tau\to-\infty}\mathrm{Scal}|<\infty$ is already satisfied. The condition $|\lim_{\tau\to\infty}\mathrm{Scal}|<\infty$ implies that $C_1=C_2=0$. Since $\beta=0$, we have $C_1=0$ and $C_2=0$ is equivalent to $c_1^2d_2^2=c_2^2d_1^2$. Inserting the conditions
$$
\beta=0,\,\,\,c_1^2d_2^2=c_2^2d_1^2
$$
into \eqref{eq:scalalongv} gives $\mathrm{Scal}(\tau)=-\frac{6c_1^2}{d_1^2}.$

Inserting $\beta=0$ and $c_1^2d_2^2=c_2^2d_1^2$ into the formula for the metric $\hat g$, a direct calculation shows that $\hat g$ has constant holomorphic sectional curvature as we claimed.

\emph{Case $\beta=-\tfrac{1}{2}$:} The condition $|\lim_{\tau\to-\infty}\mathrm{Scal}|<\infty$ implies that $C_7=0$, thus $c_1d_2^2-c_2d_1^2=0$. From $|\lim_{\tau\to\infty}\mathrm{Scal}|<\infty$ we conclude $C_1=0$, which is already satisfied. Inserting the conditions
$$
\beta=-\tfrac{1}{2},\,\,\,c_1d_2^2-c_2d_1^2=0
$$
into \eqref{eq:scalalongv} gives $\mathrm{Scal}(\tau)=-\frac{27c_1c}{2d_1^2}.$ As above, we can insert $\beta=-\tfrac{1}{2}$ and $c_1d_2^2=c_2d_1^2$ into the metric $\hat g$ and calculate that it has constant holomorphic sectional curvature as we claimed.

\emph{Case $d_1^2=d_2^2$:} This case splits into three subcases according to the sign of $\beta$.

\emph{Subcase $\beta>0$:} The condition $|\lim_{\tau\to-\infty}\mathrm{Scal}|<\infty$ is satisfied. From $|\lim_{\tau\to\infty}\mathrm{Scal}|<\infty$ it follows that $C_1=C_2=C_3=0$. This cannot be fulfilled (for example, $C_1=0$ already implies $\beta=-\tfrac{5}{2}$).

\emph{Subcase $\beta=0$:} Recall from the previously investigated case that this implies $c_1^2d_2^2=c_2^2d_1^2$ which implies that $\hat g$ has constant holomorphic sectional curvature as we wanted to show. 

\emph{Subcase $\beta<0$:} The condition $|\lim_{\tau\to-\infty}\mathrm{Scal}|<\infty$ implies $C_7=0$  {from which we obtain $\beta=-2$. Inserting the conditions
$$\beta=-2,\,\,\,d_1^2=d_2^2$$
into \eqref{eq:scalalongv} shows that $\mathrm{Scal}(\tau)=-\frac{54c^2}{d_1^2}$ and inserting it into $\hat g$, a straight-forward calculation shows that $\hat g$ has constant holomorphic sectional curvature as we claimed.} 

This completes the proof of Theorem \ref{YOconjecture} in the Liouville case.

\subsubsection{The degenerate case D2}

We proceed analogous to the last subsection and calculate the scalar curvature $\mathrm{Scal}$ of the metric {$\hat{g}$ given by formula \eqref{eq:ghat3} in the coordinates from the case D2 of Theorem \ref{thm:normalformwithv}.}

\emph{Subcase $\beta=-2$:} In this case, we obtain that $\mathrm{Scal}(\tau)$ is given by
$$
-\frac{c^2}{c_1d_1^2(c-c_1 e^{-3\tau})}\cdot\left(\left(36+d_1^2\frac{\partial^2 G}{\partial u_2^2}\right)c_1^2e^{-3\tau}-\left(18+2d_1^2\frac{\partial^2 G}{\partial u_2^2}\right)c_1c+\left(-18+d_1^2\frac{\partial^2 G}{\partial u_2^2}\right)c^2e^{3\tau}\right).
$$
We clearly have $|\lim_{\tau\to-\infty}\mathrm{Scal}|<\infty$. On the other hand, the condition $|\lim_{\tau\to\infty}\mathrm{Scal}|<\infty$ implies that $G(u_2)$ has to satisfy the \textsc{ode}
$$
d_1^2\frac{\partial^2 G}{\partial u_2^2}=18.
$$
Using this, the scalar curvature takes the form $\mathrm{Scal}(\tau)=\frac{54c^2}{d_1^2}.$ Inserting the solution $G(u_2)=\tfrac{9}{d_1^2}u_2^2+d_2u_2+d_3$ to the above \textsc{ode} into the formula for the metric $\hat g$, we obtain after a straight-forward calculation that $\hat g$ has constant holomorphic sectional curvature as we wanted to show.

\emph{Subcase $\beta\neq-2$:} Here we obtain
\begin{align}
\mathrm{Scal}(\tau)=\frac{ce^{3\tau}}{c_1d_1^2G(u_2)^3}\cdot\frac{C_0c^4+C_1c_1c^3e^{(\beta-1)\tau}+C_2c_1^2c^2e^{2(\beta-1)\tau}+C_3c_1^3ce^{3(\beta-1)\tau}+c_1^4C_4e^{4(\beta-1)\tau}}{c^2-2cc_1e^{(\beta-1)\tau}+c_1^2e^{2(\beta-1)\tau}},\label{eq:Scaldegenerate}
\end{align}
where the constants $C_0,C_1,C_2,C_3,C_4$ are given by
\begin{align}
\begin{array}{ll}
C_0=&6\beta(\beta+\tfrac{1}{2})G(u_2)^3+d_1^2\left(\left(\frac{\partial G}{\partial u_2}\right)^2-G(u_2)\frac{\partial^2 G}{\partial u_2^2}\right),\vspace{1mm}\\
C_1=&-12\beta(\beta+2)G(u_2)^3-3d_1^2\left(\left(\frac{\partial G}{\partial u_2}\right)^2-G(u_2)\frac{\partial^2 G}{\partial u_2^2}\right),\vspace{1mm}\\
C_2=&6(1+7\beta+\beta^2)G(u_2)^3+3d_1^2\left(\left(\frac{\partial G}{\partial u_2}\right)^2-G(u_2)\frac{\partial^2 G}{\partial u_2^2}\right),\vspace{1mm}\\
C_3=&-24(\beta+\tfrac{1}{2})G(u_2)^3-d_1^2\left(\left(\frac{\partial G}{\partial u_2}\right)^2-G(u_2)\frac{\partial^2 G}{\partial u_2^2}\right),\vspace{1mm}\\
C_4=&3(\beta+2)G(u_2)^3.
\end{array}\nonumber
\end{align}
We can exclude the case $\beta-1>0$. Indeed, the condition $|\lim_{\tau\to\infty}\mathrm{Scal}|<\infty$ implies $C_2=C_3=C_4=0$ but since $\beta\neq -2$, we cannot have $C_4=0$.

Suppose that $\beta-1<0$. The condition $|\lim_{\tau\to\infty}\mathrm{Scal}|<\infty$ implies  $C_0=0$, thus, $G$ has to satisfy the \textsc{ode}
\begin{align}
\frac{\partial^2 G}{\partial u_2^2}=\frac{1}{G(u_2)}\left(\left(\frac{\partial G}{\partial u_2}\right)^2+\frac{6}{d_1^2}\beta(\beta+\tfrac{1}{2})G(u_2)^3\right).\label{eq:ODEforG}
\end{align}
\weg{In the cases $\beta=0$ or $\beta=-\tfrac{1}{2}$, its solution is
$$
G(u_2)=d_2e^{d_3u_2}
$$
while for $\beta\neq0$ and $\beta\neq-\tfrac{1}{2}$, we obtain
$$
G(u_2)=\frac{1}{12}\frac{\mathrm{tanh}^2\left(\frac{u_2+d_3}{2d_1d_2}\right)-1}{d_2^2\beta(\beta+\tfrac{1}{2})}.
$$
Here, $d_2,d_3$ are certain constants.}
Using \eqref{eq:ODEforG}, we can rewrite the constants $C_1,C_2,C_3$ in the form
\begin{align}
\begin{array}{ll}
C_1=&\beta(6\beta-15)G(u_2)^3,\vspace{1mm}\\
C_2=&-(12\beta^2-33\beta-6)G(u_2)^3,\vspace{1mm}\\
C_3=&(6\beta^2-21\beta-12)G(u_2)^3.
\end{array}\label{eq:simplifiedconstants}
\end{align}
Let us evaluate further the condition $|\lim_{\tau\to\infty}\mathrm{Scal}|<\infty$.

Suppose first that $\beta>0$. We see from \eqref{eq:Scaldegenerate} that this implies $C_1=C_2=C_3=0$. But the solution $\beta=\tfrac{15}{6}$ of $C_1=0$ is already excluded since we assumed $\beta-1<0$.

Next suppose that $\beta=0$. Then, we see from \eqref{eq:Scaldegenerate} that in order to have $|\lim_{\tau\to\infty}\mathrm{Scal}|<\infty$ satisfied, we must have $C_1=C_2=0$. But as we see from  {\eqref{eq:simplifiedconstants}, $C_2$ is not zero for $\beta=0$.} It follows that the case $\beta=0$ cannot occur.  

Now suppose that $-2<\beta<0$. From \eqref{eq:Scaldegenerate} we see that $|\lim_{\tau\to\infty}\mathrm{Scal}|<\infty$ implies $C_1=0$.  {By \eqref{eq:simplifiedconstants}, $C_1=0$ is not fulfilled for $-2<\beta<0$ which henceforth excludes this case.}

For $\beta<-2$, the condition $|\lim_{\tau\to\infty}\mathrm{Scal}|<\infty$ is automatically satisfied. On the other hand,  the condition $|\lim_{\tau\to-\infty}\mathrm{Scal}|<\infty$ implies for example that $C_4$ has to vanish which is not fulfilled.

Finally, we obtain that for no choice of parameters in the case $\beta\neq -2$, the scalar curvature $\mathrm{Scal}$ is bounded. It follows that this case cannot occur. This completes the proof of Theorem \ref{YOconjecture}.

\end{document}